\documentclass{article} 

\usepackage{graphicx}
\usepackage{url}
\usepackage{placeins}
\usepackage{multirow}
\usepackage{amsfonts}
\usepackage{pdfpages}

\usepackage{amsmath}
\usepackage{amssymb}
\usepackage{amsthm}
\usepackage{float}
\usepackage[portuguese,english]{babel}
\usepackage[latin1]{inputenc}
\floatstyle{plain}
\newfloat{algorithmfloat}{thp}{ficheiro}
\floatname{algorithmfloat}{Algorithm}

\newcommand{\RR}{\ensuremath{\mathbb{R}}}
\newcommand{\NN}{\ensuremath{\mathbb{N}}}
\newcommand{\ZZ}{\ensuremath{\mathbb{Z}}}
\newcommand{\epsi}{\varepsilon}
\newcommand{\tens}[1]{{\mathbf #1}}
\newcommand{\perf}{_{\hbox{\scriptsize perf}}}
\renewcommand{\dfrac}{\displaystyle\frac}
\newcommand{\M}{\ensuremath{\mathcal{C}}}
\newcommand{\dual}[2]{\langle #1 , #2 \rangle}

\theoremstyle{plain}
\newtheorem{theorem}{Theorem}
\newtheorem{lemma}[theorem]{Lemma}
\newtheorem{proposition}[theorem]{Proposition}
\newtheorem{corollary}[theorem]{Corollary}

\theoremstyle{definition}
\newtheorem{definition}[theorem]{Definition}

\newtheorem{remark}[theorem]{Remark}

\newtheorem{algorithm}[theorem]{Algorithm}
\begin{document}

\title{A gradient-type algorithm for constrained optimization with applications \\
to multi-objective optimization of auxetic materials} 
\author{C.~Barbarosie \footnote{
CMAF-CIO, Faculdade de Ci\^encias, Universidade de Lisboa, 1749-016 Lisboa, Portugal.
\texttt{cabarbarosie@fc.ul.pt}}, 
S.~Lopes \footnote{CMAF-CIO , Faculdade de Ci\^encias, Universidade de Lisboa, 1749-016 Lisboa, Portugal and
Departmental Area of Mathematics, ISEL - Instituto Superior de Engenharia de Lisboa, Instituto Polit\'ecnico de Lisboa,
							 Rua Conselheiro Em\'idio Navarro, 1959-007 Lisboa, Portugal.
\texttt{slopes@adm.isel.pt}}, 
A.M.~Toader \footnote{
CMAF-CIO, Faculdade de Ci\^encias, Universidade de Lisboa, 1749-016 Lisboa, Portugal.
\texttt{atoader@fc.ul.pt}}}


\maketitle

\begin{abstract}
An algorithm is devised for solving minimization problems with equality
constraints.
The algorithm uses first-order derivatives of both the objective function and the
constraints.
The step is computed as a sum between a steepest-descent step (which minimizes the
objective functional) and a correction step related to the Newton method
(which aims to solve the equality constraints).
The linear combination between these two steps involves coefficients
similar to Lagrange multipliers which are computed in a natural way based on
the Newton method.
The algorithm uses no projection and thus the iterates are not feasible; the
constraints are satisfied only in the limit (after convergence).
This algorithm was proposed by one of the authors in a previous paper.
In the present paper, a local convergence result is proven for a general 
non-linear setting, where both the objective functional and the constraints 
are not necessarily convex functions.
The algorithm is extended, by means of an active set strategy,
to account also for inequality constraints and to address minimax problems.
The method is then applied to the optimization of periodic 
microstructures for obtaining homogenized elastic tensors having negative 
Poisson ratio (so-called auxetic materials) using shape and/or topology 
variations in the model hole. 
In previous works of the same authors, anisotropic homogenized tensors 
have been obtained which exhibit negative Poisson ratio in a prescribed 
direction of the plane. 
In the present work, a new approach is proposed,
that employs multi-objective optimization in order to minimize 
the Poisson ratio of the (possibly anisotropic) homogenized elastic tensor 
in several prescribed directions of the plane.
Numerical examples are presented.
\end{abstract}




Keywords : 
non-linear programming, constrained minimization, multi-objective optimization,
optimization of microstructures, porous materials, microstructure, auxetic materials




\section{Introduction}
    
We propose a numerical method for the minimization (or maximization)
of a functional, subject to constraints. 
We also present a numerical application of this method.

Section~\ref{sec: minimization-alg} is devoted to the description of 
the algorithm and to a convergence result, along with extensions to accomodate
inequality constraints and to deal with minimax problems.
The method seeks for local solutions and works in the general case, 
for non-linear and non-convex objective functional and constraints.
Smoothness of these functions is however required, since their
gradients are used.
The method is two-fold, involving a sum between a steepest-descent step 
(which minimizes the objective functional) and a correction step related to 
the Newton method (which aims to solve the equality constraints).
The linear combination between these two steps uses certain coefficients
similar to Lagrange multipliers which are computed in a natural way based on
the Newton method.
The algorithm uses no projection and thus the iterates are not feasible; the
constraints are satisfied only in the limit (after convergence).
Convergence is proven under the hypothesis that a certain Hessian-like matrix 
is positive definite at the solution.
Inequality constraints are dealt with by using an active set methodology.
Special attention is given to activation and deactivation strategies;
the deactivation criterion is entirely based on the sign of the associated
Lagrange multiplier.
The case of an infinite (continuous) family of inequality constraints is
discussed, as well as the extension of the method for treating minimax problems.

Section~\ref{sec: Poisson coef} shows an application on a large-scale example,
involving the minimization of the Poisson ratio(s) of a composite material,
in the context of homogenization theory.
To achieve this goal, we perform shape and/or topology variations in
the model hole that characterizes the microstructure.
We use the minimax algorithm in order to minimize in simultaneous the Poisson
coefficient of the (possibly anisotropic) homogenized material alog many
directions of the plane.

Some closing comments are made in Section~\ref{conclusion}.

\section{The minimization algorithm}
\label{sec: minimization-alg}

In this Section, we propose an algorithm for the minimization of a functional
subject to constraints.
The algorithm seeks for local solutions (as usual for gradient-based methods).
It deals with \emph{non-essential constraints}, that is, with constraints whose
violation does not render the problem ill-posed.
In the method here proposed, the constraints are usually violated during the
optimization process, and become satisfied only in the limit (after convergence);
see, however, subsection~\ref{sec: box} for an exception.

We begin by describing the case of equality constraints (subsection~\ref{sec: alg-eq}).
The method here described has already been used in \cite{Bar03} by one of the authors
to solve large-scale optimization problems;
in the present paper, the convergence of the algorithm is proven under
the hypothesis that a certain Hessian-like matrix is positive definite
at the solution (Theorem~\ref{th: convergence} in subsection~\ref{sec: convergence}).
Our methodology can be regarded as a gradient method applied in
the direction tangent to the manifold determined by the constraints, 
together with a Newton method applied in the orthogonal direction. 
This method, although not very fast (it has linear convergence) 
is quite natural, easy to implement, and has the advantage of requiring
solely the first derivatives of the objective and of the constraint
functions. 

		A generalization is proposed in subsections~\ref{sec: ineq} and
\ref{sec: deactiv} which deals with inequality constraints,
    based on an active-set strategy. During the optimization process,
    an inequality constraint is activated as soon as it is violated.
    Its deactivation depends on the sign of the associated Lagrange multiplier.
    To some extent, this procedure can be seen as a generalization of 
    the simplex method to nonlinear functions.
		Based on these ideas, in subsection~\ref{sec: minimax} we further extend the algorithm with minimax
		problems in sight.


    The following notation will be used: $ x\in\RR^n $ is the vector of \emph{variables}
		(also called \emph{unknowns} or \emph{parameters}); $ x_i $ will denote the components
		of $x$ while $ x^{(k)} $ will denote a sequence of vectors; $f$ is the \emph{objective
		function}, a scalar function of $x$ that we want to minimize or maximize;
		the \emph{constraints} will be modelled by a vector function $ g : \RR^n \to \RR^m $.
    The Jacobian matrix of a vector function $g$ will be denoted by $ Dg $ while its
		transpose will be denoted by $ \nabla g $.
    In particular, for a scalar function $f$, $ \nabla\! f $ will be the usual gradient.
    The Hessian matrix of $f$ will be denoted by $ D^2 f $.
Partial derivatives will be denoted by a comma, \textit{e.g.} $ g_{i,j} = 
\frac{\partial g_i}{\partial x_j} $.

We have collected in Appendix \ref{sec: unconstr} some well-known results on unconstrained
optimization, while Appendix \ref{sec: constr} gives the theoretical
background on (equality) constrained optimization.

  \subsection{A gradient algorithm for equality constrained problems}
  \label{sec: alg-eq}

Consider the minimization problem
    \[ \min_{x \in \M} f(x), \qquad \M = \{ x \in \RR^n : \; g(x) = 0
        \}. \tag{$\mathcal P$} \] 
    A typical case in structural design arises when engineers adjust the parameters (variables) to optimize the
    performance of a structure while keeping a prescribed \emph{cost}. In such a framework, the constraint function $g$
    appearing in ($\mathcal P$) is thought of as a \emph{cost function}, a scalar function that (in a broad sense)
    stands for the structure's ``price'' (or more precisely, the difference between the cost function and a prescribed
    ``price''). For presentation purposes, the discussion will be 
    initially restricted to this case of only one constraint ($ m=1 $) and
    subsequently extended to account for multiple constraints.

For the treatment of ($\mathcal P$) we will try to combine 
the ideas from Appendix \ref{sec: unconstr} with those from Appendix \ref{sec: newton}.
The algorithm should pursue two goals simultaneously :
decrease the value of $f$ while solving the equation $ g = 0 $. 
Our approach sets up a direction that targets  both goals at once, as described below.

    Given an iterate $x^{(k)}$, the next iterate will be defined by an
increment $\delta^{(k)}$, that is, $ x^{(k+1)} = x^{(k)} + \delta^{(k)} $.
The increment $\delta^{(k)}$ will be
the sum of two components: one of them is the vector 
    $ - \eta\, \nabla\! f(x^{(k)}) $ (with $ \eta > 0 $ fixed) 
    corresponding to the steepest descent algorithm; 
    the other one aims at fulfilling the constraint equation
    $ g = 0 $ and has the form $ -\lambda^{(k)}\, \nabla g(x^{(k)}) $, 
    where $\lambda^{(k)}\in\RR$ is a sort of Lagrange multiplier:
\[ \delta^{(k)} = - \eta\, \nabla\! f(x^{(k)}) - \lambda^{(k)}\, \nabla g(x^{(k)})  \]
The multiplier $ \lambda^{(k)} $ is defined adaptively in a natural way,
inspired in Proposition~\ref{prop: Newton}, Appendix \ref{sec: newton}. 
It suffices to impose the Newton-type condition, relative to the equation $ g = 0 $,
\[ \dual{\nabla g(x^{(k)})}{\delta^{(k)}} = -g(x^{(k)}) \] 
which is immediately solvable: 
\begin{equation}\label{eq: lambda}
\lambda^{(k)} = \frac{g(x^{(k)}) - \eta\, \dual{\nabla g(x^{(k)})}{\nabla\! f(x^{(k)})}}{\| \nabla g(x^{(k)}) \|^2} \cdot
\end{equation}
    With this choice of the multiplier, the whole procedure amounts to
    performing a ``tangential gradient method'' to minimize $f$, together with a unidimensional 
    Newton method to solve the constraint equation $ g = 0 $.

    To better understand the last assertion, consider the following reasoning. In the neighborhood of a solution
    $x^\ast$ there are two main directions to consider from $x^{(k)}$: the direction $ \nabla g(x^{(k)}) $, orthogonal to the
    level set $ \M_k = \{ y \in \RR^n : \; g(y) = g(x^{(k)}) \} $, and the subspace orthogonal to it 
    (whose vectors are tangent to $\M_k$ at $x^{(k)}$). 
    In this latter subspace we have to minimize $f$
    (note that, since the solution $ x^\ast $ should minimize $f$ in a level
    set of $g$, $ \M $, there is no point in decreasing $f$ along directions other than tangent ones); in
    the direction of $ \nabla g(x^{(k)}) $ we want to solve the equation $ g = 0 $, moving the next iterate closer to $\M$.
	A very simple method is obtained which, somewhat surprisingly, is not mentioned in the literature.

\begin{remark}
The algorithm here proposed is somewhat similar to the Newton method described in 
\cite[Section 12.1]{BS03} with the major difference that we use information related to
the first derivatives only (of both the objective function and the constraints).
Note that there are many practical problems in which second derivatives are
impossible (or very expensive) to compute, see Section \ref{sec: Poisson coef} of
the present paper for an example.
Note also that we do not use any projection matrices (see \textit{e.g.} \cite{JR61});
we don't need to project since our iterates do not have to satisfy the constraints.
Also, we have no need of introducing penalty functions (see \textit{e.g.} \cite{JB77}).
\end{remark}


The algorithm generalizes naturally 
to vector-valued constraint functions $ g : \RR^n \to \RR^m $ (with $ m < n $).
In this case $ \lambda^{(k)} \in \RR^m $ but the
iterates are defined in a similar fashion by 
\[ x^{(k+1)} = x^{(k)} - \eta\, \nabla\! f(x^{(k)}) - \nabla g(x^{(k)})\, \lambda^{(k)}.\]
By imposing the Newton-type condition (again, inspired in 
Proposition~\ref{prop: Newton}, Appendix \ref{sec: newton})
$$
Dg(x^{(k)})\,\delta^{(k)} = - g(x^{(k)})\,,
$$
we obtain
\begin{equation}\label{eq: lambda-sys}
D g(x^{(k)})\, \nabla g(x^{(k)})\, \lambda^{(k)} = g(x^{(k)}) - \eta\, D g(x^{(k)})\, \nabla\! f(x^{(k)}).
\end{equation}
In coordinate notation :
\[ x^{(k+1)}_j = x^{(k)}_j - \eta\, f_{,j}(x^{(k)}) - \sum_{i=1}^m \lambda^{(k)}_i\, g_{i,j}(x^{(k)}),
\: 1 \leqslant j \leqslant n, \]
where
\[ \sum_{i=1}^m \sum_{j=1}^n g_{l,j}(x^{(k)})\, g_{i,j}(x^{(k)})\, \lambda^{(k)}_i = 
g_l(x^{(k)}) - \sum_{j=1}^n \eta\, g_{l,j}(x^{(k)})\, f_{,j}(x^{(k)}), \: 1 \leqslant l \leqslant m. \]
    This linear system of equations uniquely determines $\lambda^{(k)}$ 
    if $ D g(x^{(k)})$ has full rank (equal to $m$); 
see Definition~\ref{regular point} in Appendix \ref{sec: constr} 
and the comments following it. 
    Even in the case of vector-valued constraints, the method can be interpreted 
    geometrically as a steepest descent method in the directions tangent to
    $ \M_k $ combined with a Newton method in the directions normal to $ \M_k $.

   \begin{algorithm}\label{alg: eq}~\\
      \noindent INPUT: initial guess $x^{(0)}$, step size $ \eta > 0 $, 
        tolerance $ \epsi > 0 $, maximum number of iterations $N$. \\
      \noindent OUTPUT: approximate solution $x$ or message of failure. \\
      \noindent\textbf{Step 1} With $k$ from $1$ to $N$, do Steps 2--5. \\
      \indent\textbf{Step 2} Compute $ \lambda $ by solving $ Dg(x^{(0)}) \,
             \nabla g(x^{(0)})\, \lambda = g(x^{(0)}) - \eta\, Dg(x^{(0)})\, \nabla\! f(x^{(0)}) $. \\
      \indent\textbf{Step 3} Set $ x = x^{(0)} - \eta\, \nabla\! f(x^{(0)}) - \nabla g(x^{(0)})\, \lambda $. \\
      \indent\textbf{Step 4} If $ \| x - x^{(0)} \| < \epsi $ then OUTPUT($x$); \\
      \indent\textbf{\phantom{Step 5}} \phantom{If $ \| x - x^{(0)} \| < \epsi $ then} STOP. \\
      \indent\textbf{Step 5} Set $ x^{(0)} = x $. \\
      \noindent\textbf{Step 6} OUTPUT('The method failed after $N$ iterations.'); \\
      \noindent\textbf{\phantom{Step 7}} STOP.
    \end{algorithm}

  \subsection{Convergence results}
  \label{sec: convergence}

    We now state and prove the main theorem regarding the method proposed in subsection~\ref{sec: alg-eq} above.
This result has been presented in the preprint \cite{BL11}.

\begin{theorem}\label{th: convergence}
Let $ f : \RR^n \to \RR $ and $ g : \RR^n \to \RR^m $ ($ m < n $) be twice continuously differentiable functions.
Let $ x^\ast \in \M $ be a regular point (see Definition~\ref{regular point} 
in Appendix \ref{sec: constr}) satisfying the KKT conditions (\ref{th: kkt}) and such that 
the matrix $ H^\ast = D^2 f(x^\ast) + \sum_{i=1}^m \lambda^\ast_i D^2{g_i}(x^\ast)$
is positive definite on $\mathcal{T}_{x^\ast}$ (see Theorem~\ref{th: existence} 
in Appendix \ref{sec: constr}).
Then there exists $ r > 0 $ such that, given $ x^{(0)} \in \bar{B}_r(x^\ast) $, the sequence of iterates defined by
\begin{equation}\label{eq: iterate}
x^{(k+1)} = x^{(k)} - \eta\, \nabla\! f(x^{(k)}) - \nabla g(x^{(k)})\, \lambda^{(k)}, 
\quad k \in \NN_0,
\end{equation}
with $\lambda^{(k)}$ determined by \eqref{eq: lambda-sys}, converges linearly to 
$x^\ast$ for sufficiently smal step lengths $ \eta > 0 $.
\end{theorem}

    The reasoning follows the same patern of the proof of Theorem~\ref{th: descent}
in Appendix \ref{sec: unconstr}, but a bit more care will have to
    be exercised in this case. First of all, an auxiliary result is established.
    \begin{lemma}\label{lem: radius}
Let $ P \not\equiv 0 $ be an orthogonal projection on $\RR^n$. 
If $ A \not\equiv 0 $ is a self-adjoint linear
operator on $\RR^n$, then $ v \neq 0 $ is an eigenvector of $PA$, 
associated with the eigenvalue $ \mu \neq 0 $, if and only if
      \begin{itemize}
        \item[$\boldsymbol{(i)}$]  $ v \in \text{Ran}(P) $,
        \item[$\boldsymbol{(ii)}$] $ (A - \mu I)v \in \text{Ker}(P) $.
      \end{itemize}
      Hence, the following estimate of the spectral radius holds: $ \rho(PA) \leqslant \rho( A\big|_{\text{Ran}(P)} ) $.
    \end{lemma}

    \begin{proof}[\bf \textit{Proof}]
      The ``if'' part of the assertion is trivial. The ``only if'' part follows basically from the fact that, $P$ being
      an orthogonal projection, one has the direct sum decomposition $ \RR^n = \text{\it Ker}(P) \oplus \text{\it
      Ran}(P) $. Hence, given an eigenpair $ u \neq 0 $ and $ \mu \neq 0 $ of $PA$, there are unique $ v \in
      \text{\it Ker}(P) $ and $ w \in \text{\it Ran}(P) $ such that $ Au = v + w $; but then, $ PAu = \mu u $
      reads $ w = \mu u $. Therefore, it must be $ u \in \text{\it Ran}(P) $ and $ Au - \mu u = v \in \text{\it
      Ker}(P) $.

      The last estimate is now obvious, since $ \rho(PA) = \rho(PA\big|_{\text{\it Ran}(P)}) $ and the spectral radius
      of an operator is dominated by the $ \ell^2 $ norm of that same operator (recall also that $ \| P \|_2 = 1 $ and
      that the spectral radius of a self-adjoint operator equals its $\ell^2$ norm).
    \end{proof}

    \begin{remark}\label{obs: radius norms}
      Another useful result regarding spectral radii and matrix norms (whose proof can be found in
			\cite[Section 1.4]{PC90}), is that for any square matrix $A$ and $ \epsi > 0 $, there exists a natural
      norm with the property that $ \|A\| < \rho(A) + \epsi $. Adding to this fact the considerations made in
      Corollary~\ref{cor: contractivity}, Appendix \ref{sec: unconstr}, one concludes that contractivity properties of differentiable maps $ S :
      \RR^n \to \RR^n $ are essentially governed by the spectral radius of their Jacobian matrices:
      if $ \rho(D S(x)) < 1 $, it always exists a vector norm for which $S$ is locally contractive around $x$.
    \end{remark}

    \begin{proof}[\bf \textit{Proof of Theorem~\ref{th: convergence}}]
We begin by rewriting the algorithm to display its fixed point nature. 
Because $x^\ast$ is a regular point, $D g(x^\ast)$ has full rank and the same is true 
for $ D g(x) $ with $x$ nearby $x^\ast$.
Thus, equation \eqref{eq: lambda-sys} has a unique solution
\begin{equation}\label{eq: gen-lambda}
\lambda^{(k)} = [D g(x^{(k)})\, \nabla g(x^{(k)})]^{-1} [g(x^{(k)}) - \eta\, D g(x^{(k)})\, \nabla\! f(x^{(k)})]\,.
\end{equation}
Putting this expression into \eqref{eq: iterate} yields
\[ \begin{aligned}
 x^{(k+1)} = x^{(k)} &- \eta \underbrace{\left[ I - \nabla g(x^{(k)})\, [D g(x^{(k)})\, \nabla g(x^{(k)})]^{-1}
 D g(x^{(k)}) \right]}_{P(x^{(k)})} \nabla\! f(x^{(k)}) \\
 &- \underbrace{\nabla g(x^{(k)})\, [D g(x^{(k)})\, \nabla g(x^{(k)})]^{-1}}_{K(x^{(k)})} g(x^{(k)})\,;
 \end{aligned} \]
so $ x^{(k+1)} = S(x^{(k)}) $, upon defining $ S(x) = x - P(x)\, \nabla\! f(x) - K(x)\, g(x) $. 
Because $x^\ast$ is a regular point, the operator $S$ is well defined locally around 
$x^\ast$. Because of Remark~\ref{obs: radius norms}, one is left
to establish that $ \rho(D S(x^\ast)) < 1 $.
      
$K(x)$ is clearly a right inverse of $ D g(x) $ and it is not difficult to prove that $P(x)$ is the matrix
of the orthogonal projection onto the tangent subspace $ \mathcal T_x $ to $ \mathcal C_x = \{ y \in \RR^n :\;
g(y) = g(x) \} $ at $x$.
There are some trivial relations involving $P(x)$, $K(x)$ and $ D g(x)$, namely: $ K(x)\, D g(x) = I -
P(x) $, $ P(x)\, K(x) = 0 $ and $ P(x)\, \nabla g(x) = 0 $; in view of this last equality, one can write 
\[ S(x) =  x - \eta\, P(x)\, [\nabla\! f(x) + \nabla g(x)\, \lambda^\ast] - K(x)\, g(x), \] 
and it is now easy to see, due to the KKT
conditions \eqref{th: kkt} in Appendix \ref{sec: constr}, that the Jacobian matrix of $S$ at $x^\ast$ is
given by
\[ \begin{aligned}
D S(x^\ast) &= I - \eta\, P(x^\ast)\, H^\ast - K(x^\ast)\, D g(x^\ast) \\
&= I - \eta\, P(x^\ast)\, H^\ast - [I - P(x^\ast)] = P(x^\ast)\, (I - \eta H^\ast).
\end{aligned} \]
Since $ I - \eta H^\ast $ is a symmetric matrix and $P(x^\ast)$ is the orthogonal projection's matrix onto
$ \mathcal T_{x^\ast} $, precisely the subspace where $H^\ast$ is positive definite, recalling Lemma~\ref{lem: radius}
and the proof of Theorem~\ref{th: descent}, Appendix \ref{sec: unconstr}, the conclusion is now at hand.
\end{proof}

\begin{remark}
The ``true'' Lagrange multiplier $\lambda^\ast$ can be easily approximated because 
the functional expression defining $\lambda^{(k)}$, using either \eqref{eq: lambda} 
or \eqref{eq: gen-lambda} depending on the number of constraints, evaluates to 
$ \eta\, \lambda^\ast $ at $x^\ast$.
More precisely, the function
\[ \Lambda(x) = [Dg(x)\, \nabla g(x)]^{-1} [g(x) - \eta\, Dg(x)\, \nabla\! f(x)] \]
is well defined around $x^\ast$ and $ \Lambda(x^\ast) = \eta\, \lambda^\ast $ 
in view of the KKT conditions	\eqref{th: kkt}, Appendix \ref{sec: constr}; 
since $\Lambda$ is continuous (because $g$, $ Dg $ and $\nabla\! f$ all are),
for $x^{(k)}$ near $x^\ast$ we have $ \lambda^{(k)} = \Lambda(x^{(k)}) \approx 
\Lambda(x^\ast) $, that is $ \lambda^\ast \approx \eta^{-1} \lambda^{(k)} $.
\end{remark}

Note that the constraints $ g(x^{(k)}) $ converge to zero faster than the iterates,
see Remark \ref{rem: faster} in Appendix \ref{sec: newton}.
This, together with Remark~\ref{rem: distance}, implies that the distance between
			$x^{(k)}$ and the manifold $\M$	defined by the constraints converges to zero faster than the distance $\| x^{(k)} - x^\ast \|$.

It is interesting to observe that the proposed algorithm converges even on certain 
minimization problems which do \emph{not} satisfy the hypotheses of 
Theorem~\ref{th: convergence}.
This is the case of Example 12.1 in \cite{BS03}, brought
by the authors as an evidence that a good minimization algorithm should take
into account the curvature of the level set defined by the constraints,
that is, information from the second-order derivatives of the constraints.
The algorithm described in the present paper shows good convergence on
this example, although it uses information from the first derivatives only.

  \subsection{Extension to inequality constraints}
  \label{sec: ineq}

    We now consider the problem \[ \min_{x \in \M} f(x), \qquad \M = \{ x \in \RR^n :\; g(x) \leqslant 0 \} , \]
    where the inequality is to be understood componentwise: \[ \M = \{ x \in \RR^n :\; g_i(x) \leqslant 0, \:
    1 \leqslant i \leqslant m \} . \]

    The necessary optimality conditions for this sort of problem are better 
    expressed in terms of the active constraints at a solution 
    $ x^\ast \in  \M $, that is, those constraints which attain equality:
    \[ \mathcal A^\ast = \{ i \in \NN :\;  1 \leqslant i \leqslant m,\: 
      g_i(x^\ast) = 0  \} . \] 
    The KKT conditions can then be written as follows: 
    \[
			\begin{cases} \nabla f(x^\ast) + \sum_{i \in \mathcal A^\ast} 
      \lambda_i^\ast \nabla g_i(x^\ast) = 0, \\ 
      \begin{aligned} &g_i(x^\ast) = 0, &i \in \mathcal A^\ast, \\ 
        &g_i(x^\ast) < 0, &i \not \in \mathcal A^\ast, \\ 
        &\lambda_i^\ast \geqslant 0, &i \in \mathcal A^\ast, \\ 
        &\lambda_i^\ast = 0, &i \not \in
        \mathcal A^\ast. \end{aligned} \end{cases}
		\] 
    See, for instance, \cite[page 160]{BS03}.
    The first two equations are simply the optimality conditions
    for the equality constrained problem obtained by requiring 
    the active constraints to be zero. 
    The third condition ensures that inactive constraints are satisfied.
The last condition specifies that inactive constraints have null Lagrange 
multipliers attached; however, this condition is usually
imposed in the KKT conditions for mere convenience; the values of those multipliers
have no relevance whatsoever.
    The fourth condition is most important for practical purposes: 
    Lagrange multipliers associated with active constraints must be 
    non-negative. 
    This will be useful in order to decide when to deactivate constraints along
    the iterations.

    We propose a generalization of Algorithm~\ref{alg: eq} which
    can handle inequality constraints.
    As in Algorithm~\ref{alg: eq}, the iterates are not necessarily 
    feasible; see, however, subsection \ref{sec: box} for an exception. 
    The strategy is based on the concept of active set; 
    this means that, at each iteration, the constraints are partitioned in
    two separate groups.
    Those inequalities considered active will be treated much
    in the same manner as the equality constraints are treated in 
    Algorithm~\ref{alg: eq}.
    The inequalities considered inactive are essentially ignored.
    Obviously, the set of active indices is not constant along the
    optimization process.
Activating and deactivating inequality constraints is the central
(and difficult) point of Algorithm~\ref{alg: ineq}.

In the proposed algorithm, an inequality is activated as soon as it is
violated (step 3 in Algorithm~\ref{alg: ineq}).
The deactivation criterion is not as straightforward.
It is certainly not a good idea to deactivate a constraint as soon as it
is fulfilled again (\textit{i.e.}, when the value of $ g_i $ becomes negative again).
Recall that an active inequality constraint is treated essentially as an
equality constraint.
Recall also that in our approach the constraints are not fulfilled along
the optimization process (they are satisfied only in the limit).
So, activating and deactivating an inequality constraint on the sole criterion
of it being fulfilled or violated would often produce a
\emph{zigzagging} phenomenon (the same constraint being activated and
deactivated repeatedly).

We propose that a constraint should be kept active as long as the process of 
minimization of $f$ has the \emph{tendency} of violating that particular constraint.
In order to measure this tendency, we use the sign of the respective Lagrange 
multiplier as a criterion.
Lagrange multipliers associated to active constraints should be positive
(see the above KKT conditions).
Thus, we choose to deactivate a constraint when the associated Lagrange
multiplier becomes negative (step 6 of Algorithm~\ref{alg: ineq}).
To some extent, this procedure can be seen as a generalization of 
the simplex method to nonlinear functions.
See \cite[Section 16.5]{NW06}) for a somewhat similar strategy; note that in
\cite{NW06} a distinction is made between active constraints and a working set
of constraints, a terminology that we do not use.
See also the discussion in \cite[Section 3]{JB77} where the term ``basis'' is used
for the set of current active constraints.

The question arises as to what to do when more than one Lagrange
multiplier becomes negative at the same iteration.
Should we deactivate all the constraints corresponding to negative multipliers ?
Note that, if we deactivate one constraint, the remaining Lagrange multipliers
should be computed again, and they may change signs.
Should we deactivate only the constraint corresponding to the most negative
multiplier ?
Does it make sense to compare the value of one Lagrange multiplier to another ?
In order to fix ideas, in Algorithm~\ref{alg: ineq} we choose to deactivate
the constraint corresponding to the most negative Lagrange multiplier,
then compute again the remaining multipliers (steps 5 and 6).
A more detailed discussion of the deactivation criterion is postponed
to subsection~\ref{sec: deactiv}.

\begin{algorithm}\label{alg: ineq}~\\
 \noindent INPUT: initial guess $x^{(0)}$, step size $ \eta > 0 $,
 tolerance $ \epsi > 0 $, maximum number of iterations $N$. \\  
   \noindent OUTPUT: approximate solution $x$ or message of failure. \\
   \noindent\textbf{Step 1} Set $ \mathcal A = \varnothing $. 
     \emph{(no active constraints)} \\
   \noindent\textbf{Step 2} With $k$ from $1$ to $N$, do Steps 3--9. \\
   \indent\textbf{Step 3} With $i$ from $1$ to $m$, do \\
   \indent\indent\textbf{\phantom{Step 4}} If $ g_i(x^{(0)}) > 0 $ then set $ \mathcal A = \mathcal A \cup \{i\} $;
 \emph{(constraint $ g_i \leqslant 0 $ is being violated, thus we set it active)} \\
      \indent\textbf{Step 4} Compute $ \lambda_j $ ($ j \in \mathcal A $) by solving \\
 \indent\indent\indent $ \sum_{j \in \mathcal A} \lambda_j\, \dual{\nabla g_i(x^{(0)})}{\nabla g_j(x^{(0)})} =
       g_i(x^{(0)}) - \eta\, \dual{\nabla g_i(x^{(0)})}{\nabla\! f(x^{(0)})} $, $ i \in \mathcal A $. \\
      \indent\textbf{Step 5} Set $ i = \mathrm{arg\:min}_{j \in \mathcal A} \, \lambda_j $. \\
      \indent\textbf{Step 6} If $ \lambda_i < 0 $ then set $ \mathcal A = \mathcal A \setminus \{i\} $; 
  \emph{(constraint $ g_i \leqslant 0 $ is set inactive)} \\
      \indent\textbf{\phantom{Step 7}} \phantom{If $ \lambda_i < 0 $ then} GOTO Step 4. \\
      \indent\textbf{Step 7} Set $ x = x^{(0)} - \eta\, \nabla\! f(x^{(0)}) - \sum_{i \in \mathcal A} \lambda_i\, \nabla g_i(x^{(0)}) $.\\
      \indent\textbf{Step 8} If $ \| x - x^{(0)} \| < \epsi $ then OUTPUT($x$); \\
      \indent\textbf{\phantom{Step 9}} \phantom{If $ \| x - x^{(0)} \| < \epsi $ then} STOP. \\
      \indent\textbf{Step 9} Set $ x^{(0)} = x $. \\
      \noindent\textbf{Step 10} OUTPUT('The method failed after $N$ iterations.'); \\
      \noindent\textbf{\phantom{Step 10}} STOP.
    \end{algorithm}

    Convergence proofs for such methods usually assume some 
    idealized procedure that is hardly employed in practice.
    We prefer not to state any kind of convergence result.
    In general, convergence cannot be guaranteed and \emph{zigzagging}%
    \footnote{The set of active constraints changes many times.} 
    can sometimes occur, although experience shows it to be a rare phenomenon.

      Step 4 of Algorithm~\ref{alg: ineq} can become very heavy 
      if many constraints are active (and thus many Lagrange multipliers
      must be computed). 
      Note that the number of active constraints cannot exceed 
      the number of variables, so this can only happen for
      a large number of variables.
      In subsection~\ref{sec: box} below, we describe how the computational
      burden associated with many active constraints 
      can be significanlty alleviated in a specific particular case.

Finally, let us note that it is not difficult to combine Algorithms~\ref{alg: eq} 
and \ref{alg: ineq} in order to treat the case when equality constraints are present 
together with inequality constraints.
Simply, the equality constraints should be kept always active.

\subsection{The case of box-like constraints}
\label{sec: box}

We now turn our attention to constraints of the simple form $ g_i(x) = a_i-x_i $
or $ g_i(x) = x_i-b_i $; they confine the vector variable $x$
to a rectangular box in $ \RR^n $. 
Due to their particular form, constraints of this type deserve a special treatment.
First, it is very easy to make a projection for such inequalities.
So, they can be treated as \emph{essential} constraints by performing a projection
as soon as they are violated (in step 3),
that is, by setting $ x_i = a_i $ or $ x_i = b_i $ respectively.

Second, their gradient has only one non-zero component (in the variable $ x_i $).
Because of this, the corresponding Lagrange multipliers can be eliminated from
the linear system in step 4 of Algorithm~\ref{alg: ineq}.
Thus, it suffices to compute the other Lagrange multipliers (if any) by
solving a reduced system of linear equations.
Then, the Lagrange multipliers associated to box-like constraints can be
recovered one by one without computational effort.
We shall explain this process in some detail.

Suppose there are $ m_1 + m_2 $ active constraints;
suppose that the first $ m_1 $ of them are of box-type,
either of the form $ x_i \geqslant a_i $ or $ x_i \leqslant b_i $.
Note that this implies that each of these box-type constraints
corresponds to a certain variable; of course it is impossible for
both $ x_i \geqslant a_i $ and $ x_i \leqslant b_i $ to become active
simultaneously for the same variable $ x_i $.

The system of linear equations defining the Lagrange multipliers
$ \lambda_j $ in step 4 of Algorithm~\ref{alg: ineq} writes
\[ \sum_{j=1}^{m_1+m_2} \lambda_j\, \dual{\nabla g_i}{\nabla g_j} =
g_i - \eta\, \dual{\nabla g_i}{\nabla\! f} \,,\ 1\leqslant i\leqslant m_1+m_2 \]
We treat differently the first $ m_1 $ unknowns and the first $ m_1 $ equations :
\[ \begin{aligned}
\sum_{j=1}^{m_1} \lambda_j\, \dual{\nabla g_i}{\nabla g_j} \,{}+
\sum_{k=m_1+1}^{m_1+m_2} \lambda_k\, \dual{\nabla g_i}{\nabla g_k} &=
g_i - \eta\, \dual{\nabla g_i}{\nabla\! f} \,,& 1\leqslant i\leqslant m_1 \\
\sum_{j=1}^{m_1} \lambda_j\, \dual{\nabla g_l}{\nabla g_j} \,{}+
\sum_{k=m_1+1}^{m_1+m_2} \lambda_k\, \dual{\nabla g_l}{\nabla g_k} &=
g_l - \eta\, \dual{\nabla g_l}{\nabla\! f} \,,& m_1+1\leqslant l\leqslant m_1+m_2 
\end{aligned} \]
Since $ 1\leqslant i\leqslant m_1 $, we have that $ \nabla g_i = \pm e_i $,
$ e_i $ being the canonical basis in $ \RR^n $.
To fix ideas, we suppose that $ \nabla g_i = e_i $; since
$ 1\leqslant j\leqslant m_1 $, we have $ \nabla g_j = e_j $.
Also, note that $ g_i=0 $ (when the box-like constraints are activated,
a projection operation is performed, thus they are satisfied exactly)
Thus, the linear system writes
\[ \begin{aligned}
\lambda_i \,{}+ \sum_{k=m_1+1}^{m_1+m_2} \lambda_k\, g_{k,i} &=
- \eta\, f_{,i} \,,& 1\leqslant i\leqslant m_1 \\
\sum_{j=1}^{m_1} \lambda_j\, g_{l,j} \,{}+
\sum_{k=m_1+1}^{m_1+m_2} \lambda_k\, \dual{\nabla g_l}{\nabla g_k} &=
g_l - \eta\, \dual{\nabla g_l}{\nabla\! f} \,,& m_1+1\leqslant l\leqslant m_1+m_2 
\end{aligned} \]
By using the first $ m_1 $ equations, we easily express each $ \lambda_i $
in terms of $ \lambda_k $ ($ m_1+1\leqslant k \leqslant m_1+m_2 $).
By plugging these expressions into the second part of the system, we get
\[
- \sum_{j=1}^{m_1}\, \biggl(\sum_{k=m_1+1}^{m_1+m_2} \lambda_k\, g_{k,j} +
\eta\, f_{,j}\biggr)\, g_{l,j}
\,{}+ \sum_{k=m_1+1}^{m_1+m_2} \lambda_k\, \dual{\nabla g_l}{\nabla g_k} =
g_l - \eta\, \dual{\nabla g_l}{\nabla f} \,,\ m_1+1\leqslant l\leqslant m_1+m_2 
\]
and thus
\[ \sum_{k=m_1+1}^{m_1+m_2} \lambda_k\, \dual{\nabla g_l}{\nabla g_k} \,{}-
\sum_{k=m_1+1}^{m_1+m_2} \sum_{j=1}^{m_1} \lambda_k\, g_{k,j}\, g_{l,j} =
g_l - \eta\, \dual{\nabla g_l}{\nabla\! f} + \sum_{j=1}^{m_1} \eta\, f_{,j}\,g_{l,j} \,,
\ m_1+1\leqslant l\leqslant m_1+m_2 \]

By expanding the inner products between gradients, we rewrite the above as
\[ \sum_{k=m_1+1}^{m_1+m_2} \sum_{j=m_1+1}^{n} \lambda_k\, g_{k,j}\, g_{l,j} =
g_l - \eta\, \sum_{j=m_1+1}^{n} f_{,j}\,g_{l,j} \,,
\ m_1+1\leqslant l\leqslant m_1+m_2 \]
or, equivalently,
\[ \sum_{k=m_1+1}^{m_1+m_2} \lambda_k\, \dual{\nabla^*\! g_l}{\nabla^*\! g_k} =
g_l - \eta\, \dual{\nabla^*\! g_l}{\nabla^*\! f} \,,
\ m_1+1\leqslant l\leqslant m_1+m_2 \]
where the symbol $ \nabla^* $ denotes the gradient of the respective function
\emph{with respect to the last $ n-m_1 $ variables only}.
Also, the symbol $ \dual\cdot\cdot $ in the above system of linear equations
represents the inner product in $ \RR^{n-m_1} $ and not in $ \RR^n $ as
in the previous formulae.

Based on the above considerations, Algorithm \ref{alg: ineq} can be reformulated
as follows (we denote by $ \mathcal B $ the set of box-like constraints).

\begin{algorithm}\label{alg: box}~\\
 \noindent INPUT: initial guess $x^{(0)}$, step size $ \eta > 0 $,
 tolerance $ \epsi > 0 $, maximum number of iterations $N$. \\  
   \noindent OUTPUT: approximate solution $x$ or message of failure. \\
   \noindent\textbf{Step 1} Set $ \mathcal A = \varnothing $. 
     \emph{(no active constraints)} \\
   \noindent\textbf{Step 2} With $k$ from $1$ to $N$, do Steps 3--9. \\
   \indent\textbf{Step 3} With $i$ from $1$ to $m$, do \\
   \indent\indent\textbf{\phantom{Step 4}} If $ g_i(x^{(0)}) > 0 $ then set $ \mathcal A = \mathcal A \cup \{i\} $;
 \emph{(constraint $ g_i \leqslant 0 $ is being violated, thus we set it active)} \\
   \indent\indent\textbf{\phantom{Step 4}} If $ i\in \mathcal B $ then set $ x^{(0)}_i = a_i $ or $ b_i $;
 \emph{(we project)} \\
      \indent\textbf{Step 4} Compute $ \lambda_j $ ($ j \in \mathcal A\setminus\mathcal B $) by solving \\
 \indent\indent\indent $ \sum_{j \in \mathcal A\setminus\mathcal B} \lambda_j\, \dual{\nabla^*\! g_i(x^{(0)})}{\nabla^*\! g_j(x^{(0)})} =
       g_i(x^{(0)}) - \eta\, \dual{\nabla^*\! g_i(x^{(0)})}{\nabla^*\! f(x^{(0)})} $, $ i \in \mathcal A\setminus\mathcal B $. \\
      \indent\textbf{Step 5} Compute $ \lambda_j $ ($ j\in\mathcal A \cap\mathcal B $) as $ \lambda_j = - \eta\, f_{,j}(x^{(0)}) - \sum_{k\in\mathcal A\setminus\mathcal B} \lambda_k\, g_{k,j}(x^{(0)})  $ \\
      \indent\textbf{Step 6} Set $ i = \mathrm{arg\:min}_{j \in \mathcal A} \, \lambda_j $. \\
      \indent\textbf{Step 7} If $ \lambda_i < 0 $ then set $ \mathcal A = \mathcal A \setminus \{i\} $; 
  \emph{(constraint $ g_i \leqslant 0 $ is set inactive)} \\
      \indent\textbf{\phantom{Step 7}} \phantom{If $ \lambda_i < 0 $ then} GOTO Step 4. \\
      \indent\textbf{Step 8} Set $ x = x^{(0)} - \eta\, \nabla\! f(x^{(0)}) - \sum_{i \in \mathcal A} \lambda_i\, \nabla g_i(x^{(0)}) $.\\
      \indent\textbf{Step 9} If $ \| x - x^{(0)} \| < \epsi $ then OUTPUT($x$); \\
      \indent\textbf{\phantom{Step 9}} \phantom{If $ \| x - x^{(0)} \| < \epsi $ then} STOP. \\
      \indent\textbf{Step 10} Set $ x^{(0)} = x $. \\
      \noindent\textbf{Step 11} OUTPUT('The method failed after $N$ iterations.'); \\
      \noindent\textbf{\phantom{Step 11}} STOP.
    \end{algorithm}


Recall the tricky detail that ``blocked'' variables $ x_i $ with 
$ i\in \mathcal A \cap \mathcal B $
should be ignored when computing the scalar products between gradients
is step 4 of Algorithm \ref{alg: box}; 
they should also be left unchanged in step 8 since a projection has been
performed previously (in step 3).
In a word, those ``blocked'' variables must be treated
as if they were no longer variables but mere parameters,
equal to $ a_i $ or to $ b_i $.
However, Lagrange multipliers associated to constraints in 
$ \mathcal A \cap \mathcal B $
are meaningful (and are used in steps 6 and 7 to decide deactivation).


The technique above described has been successfully used in \cite{BLcompl}.

\subsection{Activation and deactivation strategies}
\label{sec: deactiv}

Different criteria for deactivating constraints (steps 5 and 6
of Algorithm~\ref{alg: ineq}, steps 6 and 7
of Algorithm~\ref{alg: box}) may be considered when
more than one Lagrange multiplier becomes negative at a certain
step of the algorithm.
For instance, one could deactivate at once all the constraints 
with negative multipliers instead of deactivating only the most negative one.

We suggest that two different situations should be distinguished.
In the first one, which we shall describe as \emph{discrete constraints}, 
there is a relatively small number of inequality constraints.
These constraints may be very different of each other.
They may have different physical nature, perhaps different physical units and
different orders of magnitude.
It makes no sense to compare their values, and it makes no sense to compare
their associated Lagrange multipliers.
Thus, there is no point in choosing the ``most negative'' multiplier, as
done in step 5 of Algorithm~\ref{alg: ineq} (step 6 of Algorithm~\ref{alg: box}).
We consider that in this case at most one Lagrange multiplier should become
negative at each step of the algorithm.
The event of more than one Lagrange multiplier becoming negative
at a certain step should be interpreted as a warning that the optimization
process is going too fast.
Perhaps the value of the parameter $ \eta $ should be decreased.
Thus, in this case, steps 5 and 6 of the Algorithm~\ref{alg: ineq}
(steps 6 and 7 of Algorithm~\ref{alg: box}) should be
reformulated in order to test whether more than one multiplier is negative,
and to take appropriate measures if this happens.

Actually, the above considerations apply also to the activation of
constraints.
If more than one constraint is violated at one step of the algorithm,
this again should be viewed as a warning that the optimization process
is going too fast.

In the second situation, there are many inequality constraints, very similar to
each other.
They share the same physical units and have the same order of magnitude.
In a word, they are comparable.
We shall call such a set of constraints an \emph{almost-continuum of constraints}
because this may appear, for instance, as a discretization of the continuum case
described in subsection~\ref{sec: contin}.
Note that this is not the same as the situation described in 
subsection~\ref{sec: box}, where there are many constraints but they are 
very different from each other (they constrain different variables $ x_i $).

In the case of an almost-continuum of constraints, 
even before discussing the deactivation strategy, we should
take a look at how activation is done, that is, at how constraints are violated.
Since there are many constraints, close to each other, if one of them is
violated we expect the ``neighbour'' ones, that is, ones which are similar to it,
to be violated, too.
So, ``group violations'' are to be expected in this second situation.
If this happens, there is no point in activating the whole group of constraints.
We propose that only the ``worst'' one should be activated, that is, the one
which becomes more positive (recall that the constraints are of the form
$ g_i \leqslant 0 $).
In subsequent steps, the algorithm will push the value of that ``worse''
$ g_i $ towards zero, and this will have the side effect of pushing also 
its ``neighbour'' constraints towards zero.

It is not an easy task to implement the above described ideas into a computer program.
Clusters of violated constraints must be identified and monitored along the optimization process.
Each cluster should have a ``leader'' (the most violated constraint in that
cluster) which is active.
By pushing the ``leader'' towards zero one hopes to control the behaviour of the
whole cluster.
Along the optimization process, clusters may merge or split, which makes this 
programming task really challenging.

The implementation of the above ideas is only possible if the (almost-continuous)
set of constraints has some internal organization which allows us to identify
the closest neighbours of a given constraint, like the one described in 
Remark \ref{rem: poisson ring} for instance.
Using this notion of vicinity, we suggest the following activation strategy.
Among the set of all violated constraints, one should only activate those which
are ``more violated'' than all of its neighbours (these will be the ``leaders'').
Along the optimization process, one should keep checking if the neighbours of
the ``leaders'' become ``more violated'' than the ``leader'' and 
switch the activation flag towards the ``most violated'' one.

Going back to the deactivation issue, we see that the case of an almost-continuum
of constraints is actually not very
different from the case of discrete constraints, since only the ``leaders'' 
of the clusters of constraints have been activated.
Thus, the active constraints continue to be few and ``far'' from each other
(that is, different), and the same deactivation strategy should be employed
as in the first situation.

Finally, note that a combination of the two situations (discrete constraints 
and almost-continuum constraints) may appear in some examples.
It is not difficult to adapt the algorithm in order to deal with such problems.
Even several almost-continua of constraints can be treated in the same manner.

\subsection{A continuum of constraints}
\label{sec: contin}

A true continuum of constraints may also be of interest for certain problems.
In the sequel, we shall use the notation $ g_\xi $ instead of $ g_i $,
$ \xi $ being a parameter indexing the family of constraints.

For instance, in the example studied in section~\ref{sec: Poisson coef}
it would be interesting to allow for any angle between 0 and 180 degrees
(this would give rise to a one-dimensional continuum of constraints).
Another example is the optimization of a structure subject to an incoming
wave; the wave may come from any direction of the plane and may have any 
frequency within a certain range
(this would be a two-dimensional continuum of constraints).
Often, the optimization of a structure subject to multiple loads also falls 
into this category.

This case of a continuum of constraints can be dealt with, 
more or less in the same manner as described
in subsections~\ref{sec: ineq} and \ref{sec: deactiv}.
The main novelty is that we should now actively seek for the worst case
(the most violated constraint) within a continuum, and this should be done
by means of a maximization algorithm.
This may be viewed as a dual approach: on one hand, we minimize (in $x$) 
a function $ f(x) $ while on the other hand we look for the worst case by 
maximizing (in $ \xi $) $ g_\xi(x) $.

Note that we have not implemented these ideas yet; the example given in 
Section~\ref{sec: Poisson coef} has been solved using the discrete approach
described in subsections~\ref{sec: ineq} and \ref{sec: deactiv}.

\subsection{Extension to minimax problems}
\label{sec: minimax}

We now turn to the problem of minimizing simultaneously a family of functionals.
More precisely, the goal is to minimize the maximum of several different functionals:
\[
\min_x \max\, \{ f_1(x), f_2(x), \dots, f_m(x) \},
\]
where all of the $f_i$ are smooth functions (making it tempting to try and bypass the
non-smoothness of the inner max-function in some way).
It is easy to re-write this problem into a form appropriate to be treated by
the method described in subsections~\ref{sec: ineq} and \ref{sec: deactiv}.
It suffices to introduce a new (fake) variable $z$ and minimize (in $x$ and
$z$) the function $ F(x,z) = z $ subject to the constraints $ f_i(x) \le z $,
$ i = 1, 2,\dots, m $.

The case of a continuous minimax problem can be dealt with in the same way.
The problem
\[
\min_x \max_\xi f_\xi(x)
\]
can be reformulated as the minimization of $ F(x,z) = z $
subject to $ f_\xi(x) \le z $, $ \forall\xi $.
Thus, the method described in subsection~\ref{sec: contin} can be applied.

\section{Application of the algorithm to the optimization of auxetic materials}
\label{sec: Poisson coef}

We now show an application of the minimization algorithm described in 
subsection~\ref{sec: minimax} to the optimization of macroscopic properties of
periodic microstructures, in the context of linearized elasticity.
With the goal of designing a composite material having negative Poisson ratio
along all directions, we use the algorithm for minimizing in simultaneous the
Poisson ratio of the composite along many directions of the plane ($ 10 $
directions in the first example, $ 18 $ directions in the second example).
These results have been presented in the preprint \cite{BarToa13}.

A composite material will be described as a periodic microstructure,
that is, a linearly elastic body whose material coefficients vary
at a microscopic scale, according to a periodic pattern.
Homogenization theory allows one to accurately describe the macroscopic
behaviour of such a microstructure by means of so-called cellular problems,
which are elliptic PDEs subject to periodicity conditions.
Porous materials, that is, bodies with periodic infinitesimal perforations,
can be described in a similar manner.


\subsection{The cellular problem}
\label{cell-prob}

We shall consider a model hole, which is a compact set $ T\subset Y $ 
(see Figure~\ref{fig:hole}), where $Y$ is the periodicity cell.
Usually, $Y$ is the unit cube in $ \RR^n $; see, however, \cite{Bar08a}
for a general notion of periodicity.

\begin{figure}[ht]
 \centering{\hspace{8cm}\includegraphics[scale=0.7]{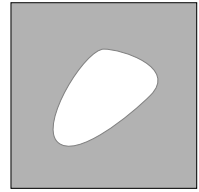} }
 \vspace{1cm}	
\caption{Periodicity cell with model hole (zoomed)}
\label{fig:hole}
\end{figure}

The perforated body is obtained by removing from $ \RR^n $ translations
of the model hole.
For a cubic cell $Y$, one has (see Figure~\ref{fig:holes90})
\begin{equation} 
\RR^n\perf = \RR^n \setminus \bigcup_{\vec k\in\ZZ^n} (T+\vec k) 
\end{equation}
\vspace{-1cm}
\begin{figure}[ht]
\centering{\hspace{5cm}\includegraphics[scale=0.7]{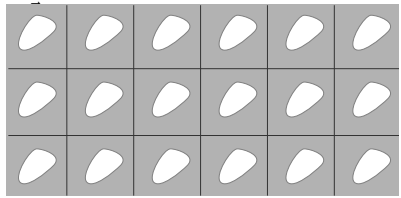} }
\vspace{1cm}
\caption{Periodically perforated plane $ \RR^2\perf$}
\label{fig:holes90}
\end{figure}

The cellular problem describing the behaviour of such
a porous material is:
\begin{equation}\label{eq:cell-pb-hole}
\left\{\begin{array}{rcl}
\hbox{find } u_{\tens A} (\vec x) & = \tens A \vec x + \phi_{\tens A}(\vec x), \\
\hbox{where } \phi_{\tens A}  & \hbox{ is a periodic function, and} \\
-\mbox{div}(\tens C \mathbf\epsi(\vec u_{\tens A})) & = \vec 0 \hbox{ in } \RR^n\perf
\\
 \tens C \mathbf\epsi(\vec u_{\tens A}) \vec n & = \vec 0 \hbox{ on } \partial T 
 \end{array}
\right.
\end{equation}

The above problem models the microscopic behaviour of
a microstructure with elastic tensor $ \tens C $,
occupying the domain $ \RR^2\perf $ and subject to 
the macroscopic strain $ \tens A $.
The homogenized elastic tensor $ \tens C^H \!\! $, describing the effective
(macroscopic) behaviour of this microstructure, is given by
\begin{equation}\label{eq:CHA-hole}
\tens C^H \tens A = \frac 1{|Y|} \int_{Y\setminus T} \tens C 
\mathbf\epsi (\vec u_{\tens A})
\end{equation}
or
\begin{equation}\label{eq:CHAB-hole}
\langle \tens C^H \tens A, \tens B\rangle = \frac 1{|Y|} 
\int_{Y\setminus T} \langle \tens C \mathbf\epsi (\vec u_{\tens A}),
\mathbf\epsi(\vec u_{\tens B}) \rangle \,,
\end{equation}
where $ \tens A $ and $ \tens B $ are given macroscopic strains.

The cellular problem (\ref{eq:cell-pb-hole}) can be
reformulated in stress, as follows (see \cite{Bar08a}) :

\begin{equation}\label{eq:stress-form-hole}
\left\{\begin{array}{rcl}
& \vec w_\sigma \in LP\perf ,\\
& -\mbox{div}(\tens C \mathbf\epsi(\vec w_\sigma)) = \vec 0 \hbox{ in } \RR^n\perf\\
& \tens C \mathbf\epsi(\vec w_\sigma) \vec n = \vec 0 \hbox{ on } \partial T \\
& \displaystyle\frac 1{|Y|} \int_{Y\setminus T} \tens C \mathbf\epsi (\vec w_\sigma) =\mathbf\sigma \,,
\end{array}\right.
\end{equation}
where $ \mathbf\sigma $ represents an applied macroscopic stress.

We shall denote by $\tens D^H$ the homogenized compliance tensor, that is, the
inverse of $\tens C^H$.

\subsection{Shape and topology derivatives}
\label{sh-top-deriv}

The effective elastic properties of the above described porous body
can be optimized by varying the size and shape of existing holes
in the periodicity cell $Y$, and also by creating new, infinitesimal, holes.

The first approach is called shape optimization (here applied at the
cellular level).
The shape derivative describes the variation of a certain
objective functional when an infinitesimal deformation is applied to
a given geometry.
Consider $ \vec\theta: \RR^n\to \RR^n $ a vector field defining the
deformation; note that $ \vec\theta $ itself should be periodic in order
to preserve the periodic character of the microstructure under study.
Then the variation induced by this deformation in the quantity
$ \langle \tens C^H \tens A, \tens B \rangle $ is (see \cite{Bar03}
and \cite[Section 6]{Bar08a})
\begin{displaymath}
D_S \langle \tens C^H \tens A, \tens B \rangle =
\frac 1{|Y|} \int_{\partial T} \langle \tens C \mathbf\epsi(\vec u_{\tens A}),
\mathbf\epsi(\vec u_{\tens B}) \rangle \, \vec\theta\cdot \vec n
\end{displaymath}
where $ \vec n $ is the unit vector normal to the boundary of the hole $T$
and pointing inside $T$.
Assuming that $ \tens C $ is a linear isotropic elastic tensor,
$ \tens C \xi = 2 \mu \xi + \lambda (\mbox{tr}\xi) I $,
the above formula becomes
\begin{displaymath}
D_S \langle \tens C^H \tens A, \tens B \rangle =
\frac 1{|Y|} \int_{\partial T} \bigl[ 2\mu \langle \mathbf\epsi(\vec u_{\tens A}),
\mathbf\epsi(\vec u_{\tens B}) \rangle + 
\lambda \hbox{tr} (\mathbf\epsi(\vec u_
{\tens A})) \hbox{tr} (\mathbf\epsi(\vec u_{\tens B})) \bigr] \, \vec\theta\cdot \vec n
\end{displaymath}
In particular, this gives the shape derivative of the homogenized
coefficients:
\begin{equation}\label{eq:DSCHAB}
D_S C^H_{ij}  = D_S \langle \tens C^H \tens f_i, \tens f_j \rangle 
 =\frac 1{|Y|} \int_{\partial T} \bigl[ 2\mu \langle \mathbf\epsi(\vec u_{\tens f_i}),
\mathbf\epsi(\vec u_{\tens f_j}) \rangle + 
\lambda \hbox{tr} (\mathbf\epsi(\vec u_{\tens f_i})) \hbox{tr} 
(\mathbf\epsi(\vec u_{\tens f_j})) \bigr] \, \vec\theta\cdot \vec n
\end{equation}
where $ (\tens f_i)_{i=1,2,3} $ is the following basis in the space of 
symmetric matrices
\begin{displaymath}
\tens f_1 = \left[\begin{array}{rcl}1\quad & 0 \\ 0\quad & 0 \end{array}\right] \,,
\ \tens f_2 = \left[\begin{array}{rcl}0\quad & 0 \\ 0\quad & 1\end{array}\right] \,,
\ \tens f_3 = \frac 1 {\sqrt 2} \left[\begin{array}{rcl}0\quad & 1 \\ 1\quad & 0\end{array}\right]
\end{displaymath}
and $ \vec u_{\tens f_i} $ are the corresponding solutions of
the cellular problem (\ref{eq:cell-pb-hole}) with
effective strain $ \tens f_i $.

A second approach for the optimization of a structure is
topology variation (here applied at the cellular level).
It consists in drilling an infinitesimal circular hole and imposing 
zero Neumann condition on the newly created boundary.
The topological derivative describes the infinitesimal variation thus
induced in the functional $ \langle \tens C^H \tens A, \tens B \rangle $,
and depends on the location $x$ of the new hole.
It can be proven (see \cite[Section 5]{Bar08a} and \cite{Toa11}) that the
topological derivative is given by
\begin{equation}\label{eq:DTCH-2d}
\begin{array}{rcl}
 D_T \langle\tens C^H \tens A, \tens B\rangle (x) & {}=
-\displaystyle\frac\pi{|Y|} \frac{\lambda+2\mu}{\lambda+\mu} \Bigl[4\mu 
\mathbf\epsi (\vec u_{\tens A}) \mathbf\epsi(\vec u_{\tens B}) + {}\\
{} & {}+\displaystyle\frac{\lambda^2+2\lambda\mu-\mu^2}\mu\, \mbox{tr}\, \mathbf\epsi(\vec u_{\tens A})
\,\mbox{tr}\, \mathbf\epsi(\vec u_{\tens B}) \Bigr] (x)
\end{array}
\end{equation}

In \cite{Bar08b}, an algorithm was proposed for optimizing the 
microgeometry of the hole(s) in the cellular problem, with the goal of
improving certain macroscopic properties of the porous microstructure
(which is a body with periodically distributed infinitesimal perforations).
The algorithm alternates shape variations with topology variations
until a certain convergence criterion is fulfilled.
The properties to be optimized include the effective bulk modulus, 
the effective response to shear and the effective Poisson coefficient 
(see \cite[Section 6]{Bar08b}).

Both shape and topology derivatives of the homogenized compliance tensor $\tens D^H$
are obtained from the derivatives of the homogenized tensor $\tens C^H$ by :
\begin{equation}\label{deriv D^H}
D_S \tens D^H_{ijkl} =-D^H_{ij\alpha\beta}\, D_S \tens C^H_{\alpha\beta\gamma\delta} \, D^H_{\gamma\delta kl}, \quad
D_T \tens D^H_{ijkl} =-D^H_{ij\alpha\beta} \, D_T \tens C^H_{\alpha\beta\gamma\delta} \, D^H_{\gamma\delta kl}.
\end{equation}
Note that formula (\ref{deriv D^H}) uses the coordinate notatin for the
fourth-order tensors $ \tens C^H $ and $ \tens D^H $, while in (\ref{eq:DSCHAB})
the indices $i$ and $j$ are relative to the basis $ (\tens f_i)_{i=1,2,3} $ in 
the space of symmetric matrices.

\subsection{Poisson ratios and the minimax technique}
\label{Poisson-ratios}

This work focuses on the search of two-dimensional periodic microstructures
exhibiting negative Poisson ratio at the macroscopic level 
(so-called auxetic materials).
In previous works of the same authors \cite{Bar08a}, \cite{Bar08b}, 
anisotropic effective elastic tensors have been obtained which exhibit 
negative Poisson ratio in a prescribed direction of the plane (the horizontal direction),
see Figure~\ref{fig:3micro}.
In the present work, we look for periodic microstructures with
the same negative Poisson ratio among all directions in the plane.
This is done by combining the techniques described in the above 
subsection~\ref{sh-top-deriv} (for shape optimization at the cellular level) 
with the minimax algorithm described in subsection~\ref{sec: minimax} 
which ensures that the largest Poisson ratio
among many directions in the plane is being minimized.

\begin{figure}
 \center{\hspace{0.7cm}\includegraphics[height=4.5cm]{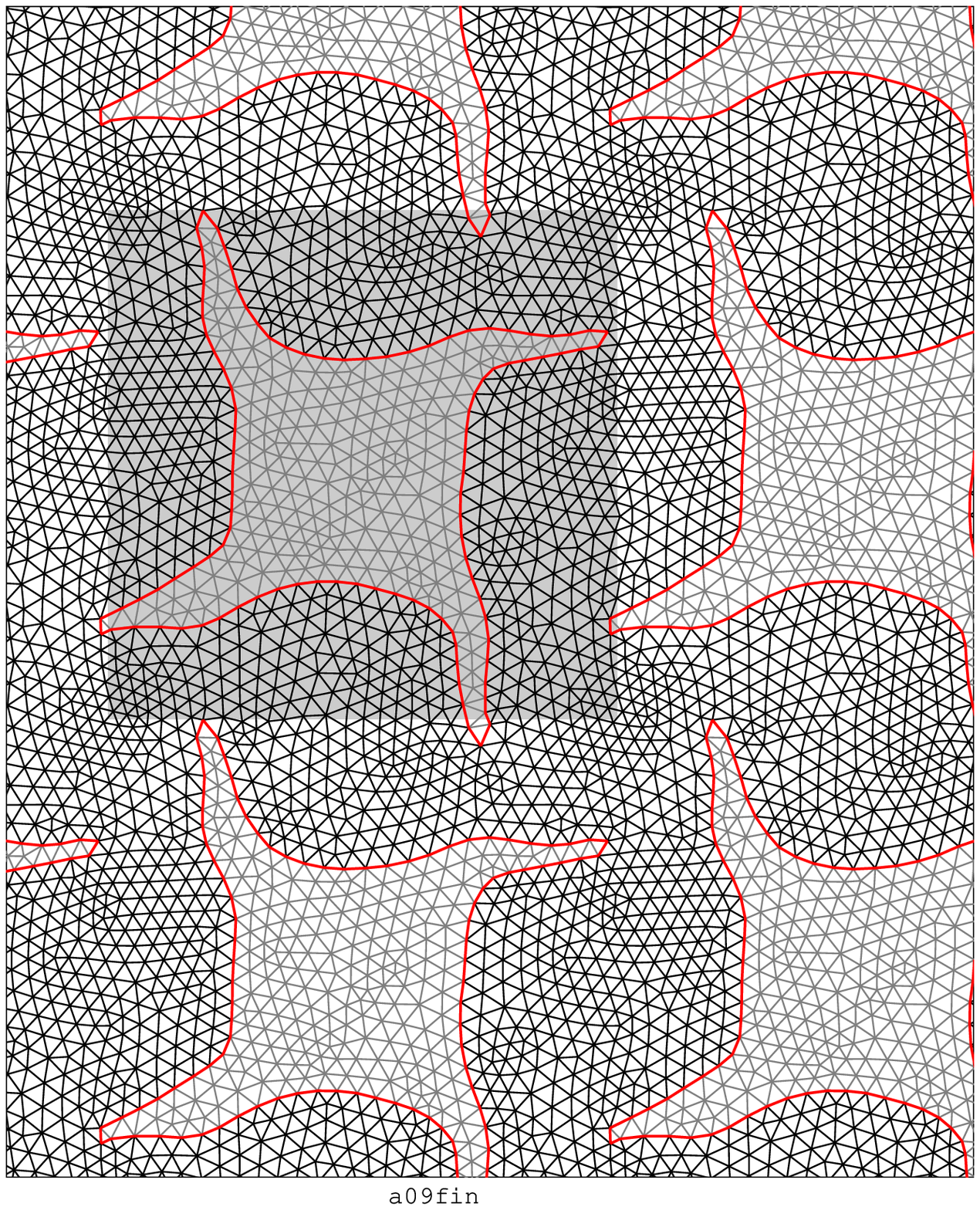}
\includegraphics[height=4.5cm]{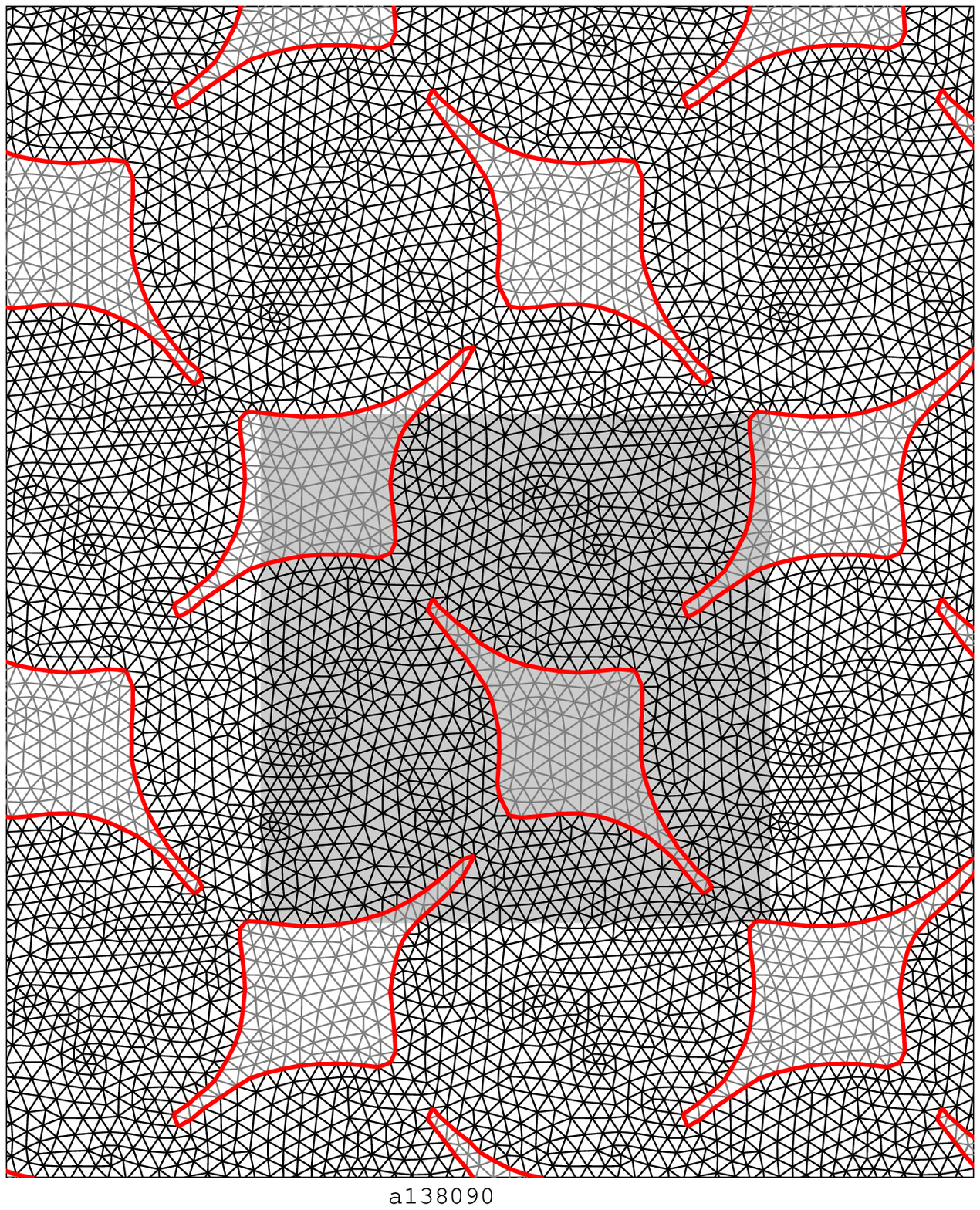}
\includegraphics[height=4.5cm]{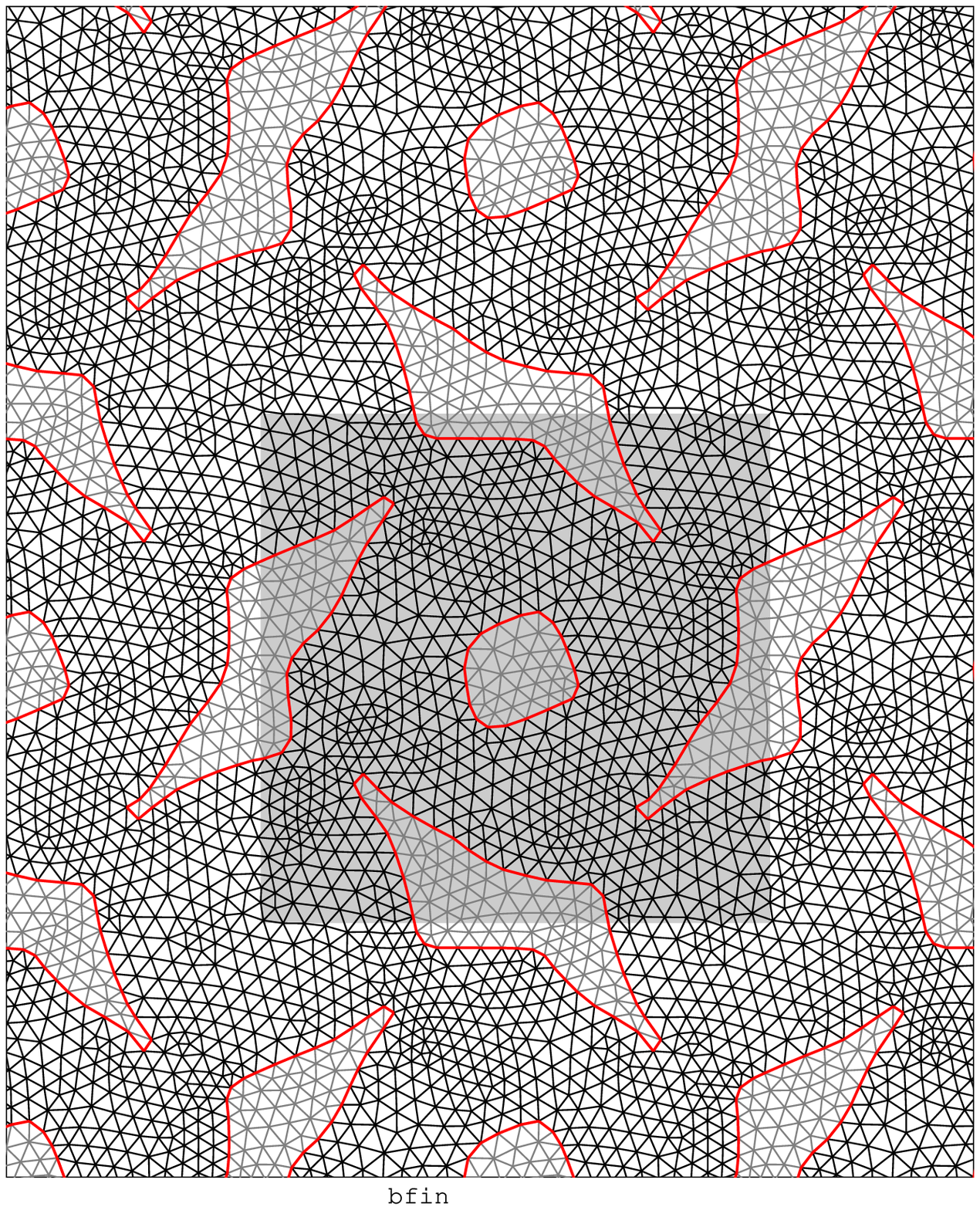} }
\vspace{1cm}
\caption{Optimized microstructures with respect to one direction only}
\label{fig:3micro}
\end{figure}

Note that the effective elastic tensor resulting from the homogenization
technique, defined by (\ref{eq:CHA-hole}) or (\ref{eq:CHAB-hole}),
is not isotropic in general.
Thus, the notion of Poisson coefficient must be defined with care : it is minus the 
ratio between the transverse strain and the axial strain when 
the material is stretched or compressed along the axial direction, see \cite{BSS93}.
In the two dimensional case under consideration, a Poisson ratio can be associated
to each unit vector $\vec{\rm v} = ({\rm v}_1, {\rm v}_2)$, arbitrarily chosen in the plane.
Consider a stretching stress applied along the direction of $\vec{\rm v}$; in the frame 
$\{\vec{\rm v}, \vec{\rm v}^\perp\}$ the stress
writes $\mathbf\sigma = \left(\begin{array}{rcl}1\quad & 0 \\ 0\quad & 0\end{array}\right)$. 
Then the Poisson ratio $\nu_{\rm v}$ in the direction $\vec {\rm v}$
is defined as $\nu_{\rm v}=-\displaystyle\frac{\varepsilon_{\perp \perp}}{\varepsilon_{\rm vv}}$, 
where $\mathbf\epsi_{\perp \perp}$
is the strain in the direction $\vec {\rm v}^\perp$ and $\mathbf\epsi_{\rm vv}$ is 
the strain in the direction $\vec {\rm v}$. 
The above defined stress $\mathbf\sigma$ expressed
in cartesian coordinates has the form 
\begin{equation}\label{sigma}
\mathbf\sigma =\left(\begin{array}{rcl}{\rm v}_1^2\quad & {\rm v}_1 {\rm v}_2 \\ {\rm v}_1 {\rm v}_2\quad & {\rm v}_2^2\end{array}\right)
\end{equation}
and the associated strain matrix is $\mathbf\epsi= \tens D^H \mathbf\sigma $ 
(recall that $\tens D^H$ is the homogenized compliance
tensor). Then the axial strain is
$\mathbf\epsi_{\rm vv}= \tens D^H \mathbf\sigma \, \vec {\rm v} \cdot \vec {\rm v} $ and 
the transverse strain is $\mathbf\epsi_{\perp\perp}= 
\tens D^H \mathbf\sigma \, \vec {\rm v}^\perp \cdot \vec {\rm v}^\perp$. The Poisson ratio writes as 
$\nu_{\rm v}=-\displaystyle\frac{\tens D^H \mathbf\sigma \, \vec {\rm v}^\perp \cdot \vec {\rm v}^\perp}
{\tens D^H \mathbf\sigma \, \vec {\rm v} \cdot \vec {\rm v}}$ and
introducing  the stress corresponding to a stretch in the direction $\vec {\rm v}^\perp$, denoted by
\begin{equation}\label{sigma.perp}
\mathbf\sigma^\perp =\left(\begin{array}{rcl}{\rm v}_2^2\quad & -{\rm v}_1 {\rm v}_2 \\ -{\rm v}_1 {\rm v}_2\quad & {\rm v}_1^2\end{array}\right),
\end{equation}
it becomes :
\begin{equation}\label{nu_v}
\nu_{\rm v}=-\displaystyle\frac{\langle \tens D^H \mathbf\sigma , \mathbf\sigma^\perp \rangle}{\langle \tens D^H \mathbf\sigma , \mathbf\sigma \rangle}.
\end{equation}
The derivative of $\nu_{\rm v}$ with respect to $\tens D^H$ is given by :
\begin{equation}\label{deriv nu_v}
\dfrac{\partial \nu_{\rm v}}{\partial \tens D^H_{ijkl}}=
-\dfrac{\sigma_{ij}\sigma^\perp_{kl}}{\langle \tens D^H \mathbf\sigma , \mathbf\sigma \rangle} 
+ \dfrac{\langle \tens D^H \mathbf\sigma , \mathbf\sigma^\perp \rangle}{\langle \tens D^H \mathbf\sigma , \mathbf\sigma \rangle^2} \sigma_{ij}\sigma_{kl}
\end{equation}
The above formulae  \eqref{nu_v} and \eqref{deriv nu_v} are suitable for implementation since they can actually be seen as depending 
on the homogenized tensor $\tens C^H$, see (\ref{deriv D^H}).

The minimax algorithm proposed in subsection~\ref{sec: minimax} will be applied to this problem.
The family of functionals to minimize simultaneously consists of all Poisson ratios associated to the effective elasticity tensor $ \tens C^H $
in many different directions of the plane. 
Specifically, we consider a (large but finite) set of directions $\vec {\rm v} $
in the plane and for each direction we associate the respective Poisson ratio $\nu_{\rm v}$ defined by \eqref{nu_v}, \eqref{sigma}
and \ref{sigma.perp}. Each Poisson ratio $\nu_{\rm v}$ is a functional of $\tens D^H$ which is the inverse tensor of 
$\tens C^H$ which in turn is function of the shape 
and topology of the perforations denoted by $T$ in subsection~\ref{cell-prob}.
The respective derivatives of these dependencies are given by formulas \eqref{deriv D^H} and \eqref{deriv nu_v}.

The algorithm will minimize the
largest Poisson ratios; this means that, after convergence, the effective 
elastic tensor thus obtained will have roughly the same Poisson ratio in 
all directions of the plane; also, if the process is successful, 
this Poisson ratio will be negative.
Note, however, that this does not ensure that $ \tens C^H $ is isotropic.

\begin{remark}\label{rem: poisson ring}
We use a (finite) family of directions $\vec {\rm v}$ indexed by an angle varying
from $ 0^{\circ} $ to $ 180^{\circ} $.
Using the terminology from subsection \ref{sec: deactiv},
this is an almost-continuous family of functionals organized as a ring
(0 degrees actually gives the same direction as 180 degrees).
Within this organization, each constraint has two closest neighbours.
\end{remark}

\subsection{Numerical implementation and numerical results}
\label{num.impl}

The algorithm used in this work is an improved version of our
home-made code, presented in \cite[Section 4]{Bar08b}.
The improvement consists in the addition of a minimax routine
which handles the optimization of the worst case among several 
functionals, as described subsection~\ref{sec: minimax}.

As explained in \cite[Section 4]{Bar08b}, in order to discretize
problem \eqref{eq:cell-pb-hole}, the microstructure is meshed with
triangular finite elements of Lagrange type of degree two.
Some of the triangles are marked as ``full'', corresponding to the
elastic material, while other triangles are marked as ``empty'',
corresponding to the hole $T$.
The interface between material and hole is marked in red (see Figure~\ref{fig:3micro}).
Note that, although for graphical purposes several contiguous cells are
represented, the mesh covers only one cell $Y$ and is ``closed'' in
itself, having no boundary.
It can be described as a mesh on the two-dimensional torus; 
the graphical representation in Figure~\ref{fig:3micro} refers to an unfolded mesh 
where vertices, segments and triangles are drawn more than once.

In order to implement the periodicity condition in \eqref{eq:cell-pb-hole},
linear+periodic functions are considered on this mesh 
(they can be identified with multi-functions on the torus).
This is done by keeping track of segments crossing the boundary of 
the cell $Y$ and by taking into account the jump of the function 
along those segments.

Along the optimization process, the mesh deforms in order for the holes
to change their shape. The deformation of the mesh is accomplished by
simply moving the vertices. However, this implies a gradual loss in
the quality of the mesh : sharp angles appear eventually, as well as
too long or too short sements. At some point, certain triangles may
even become flat or be reversed. Of course this must be prevented,
since it turns the process of solving problem \eqref{eq:cell-pb-hole}
by the finite element method ill posed and consequently unstable. 
With this end in view, the program
improves frequently the quality of the mesh, either by moving the vertices
(equilibrating the mesh) or by changing the elements of the mesh
(flipping segments, adding/eliminating vertices). See \cite[Section 5]{Bar08b}
for details.

For shape optimization, the integrands in \eqref{eq:DSCHAB} are computed.
These are scalar functions defined on the boundary of the holes and
depending on the solutions $ u_{\tens f_j} $ of three cellular problems
($ j=1,2,3 $).
A functional $J$ is chosen which depends on the homogenized coefficients
$ C^H_{ij} $ (as explained in subsection~\ref{Poisson-ratios}, here we actually
consider several functionals to be minimized simultaneously, the Poisson ratios).
The shape derivative of $J$ is computed as
\begin{displaymath}
D_S J = \sum_{ij} \frac{\partial J}{\partial C^H_{ij}}\, D_S C^H_{ij}
= \int_{\partial T} \gamma\, \vec\theta\cdot \vec n
\nonumber \end{displaymath}
where $ \gamma $ is a scalar function computed as a linear combination of
the integrands in (\ref{eq:DSCHAB}).
If a steepest descent method were used, one should choose a deformation of the interface $ \partial T $
equal to $ \vec\theta = -\gamma \vec n $ (multiplied by some positive constant $ \eta $ which
controls the speed of the process) in order to decrease the value of $J$.
Here, we use instead the method described in subsection \ref{sec: minimax} in order
to minimize simultaneously several functionals, and the shape derivative of
each of the Poisson ratios is used in Algorithm \ref{alg: box} accordingly.
The algoritm provides a desired deformation of the interface $ \partial T $, 
which is then propagated into the whole mesh by means of an averaging
process (see \cite[Section 5]{Bar08b}).
\vspace{-1cm}
\begin{figure}[ht]
\center{\hspace{1.5cm}\includegraphics[height=6cm]{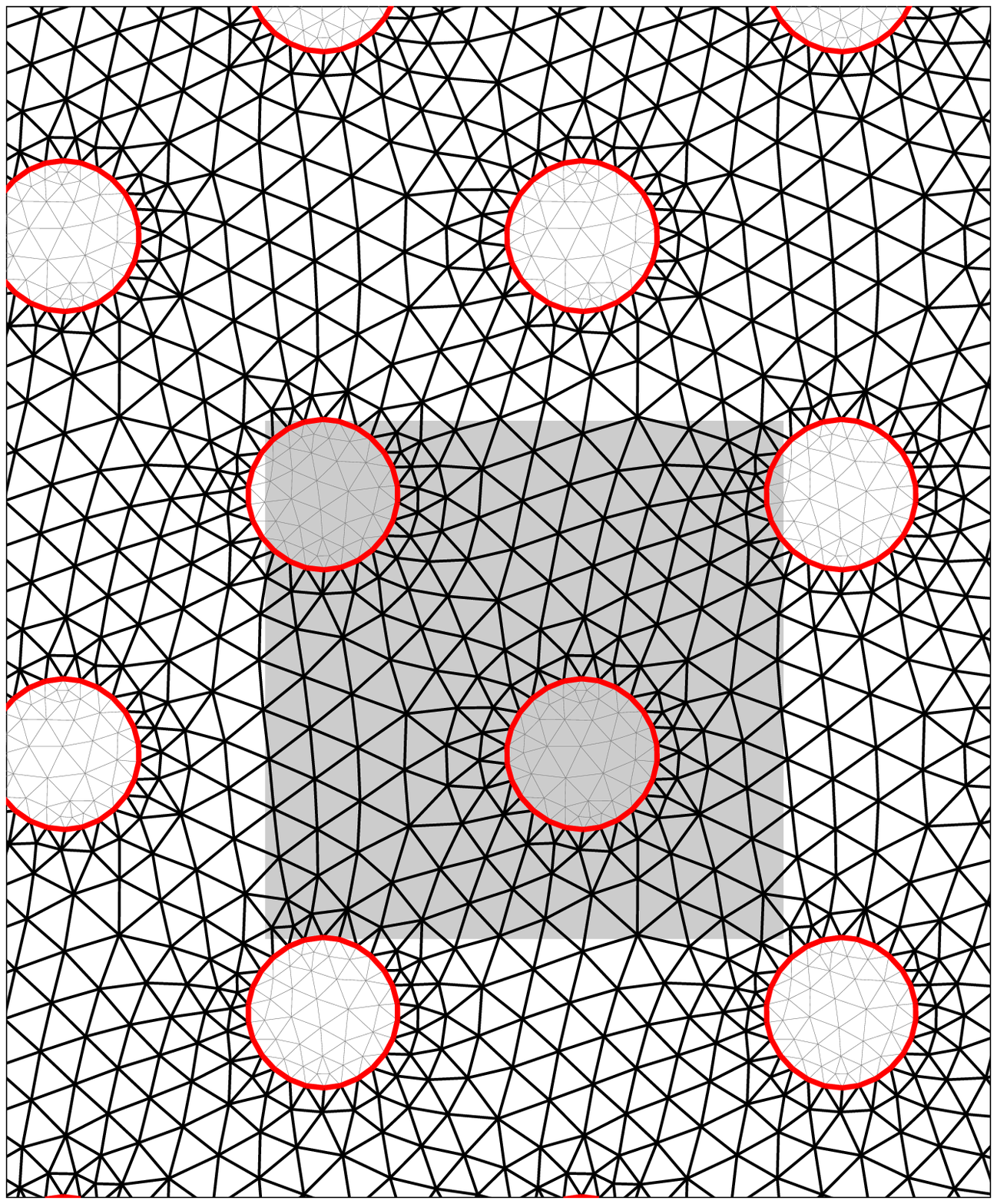} \,\,
\includegraphics[height=6cm]{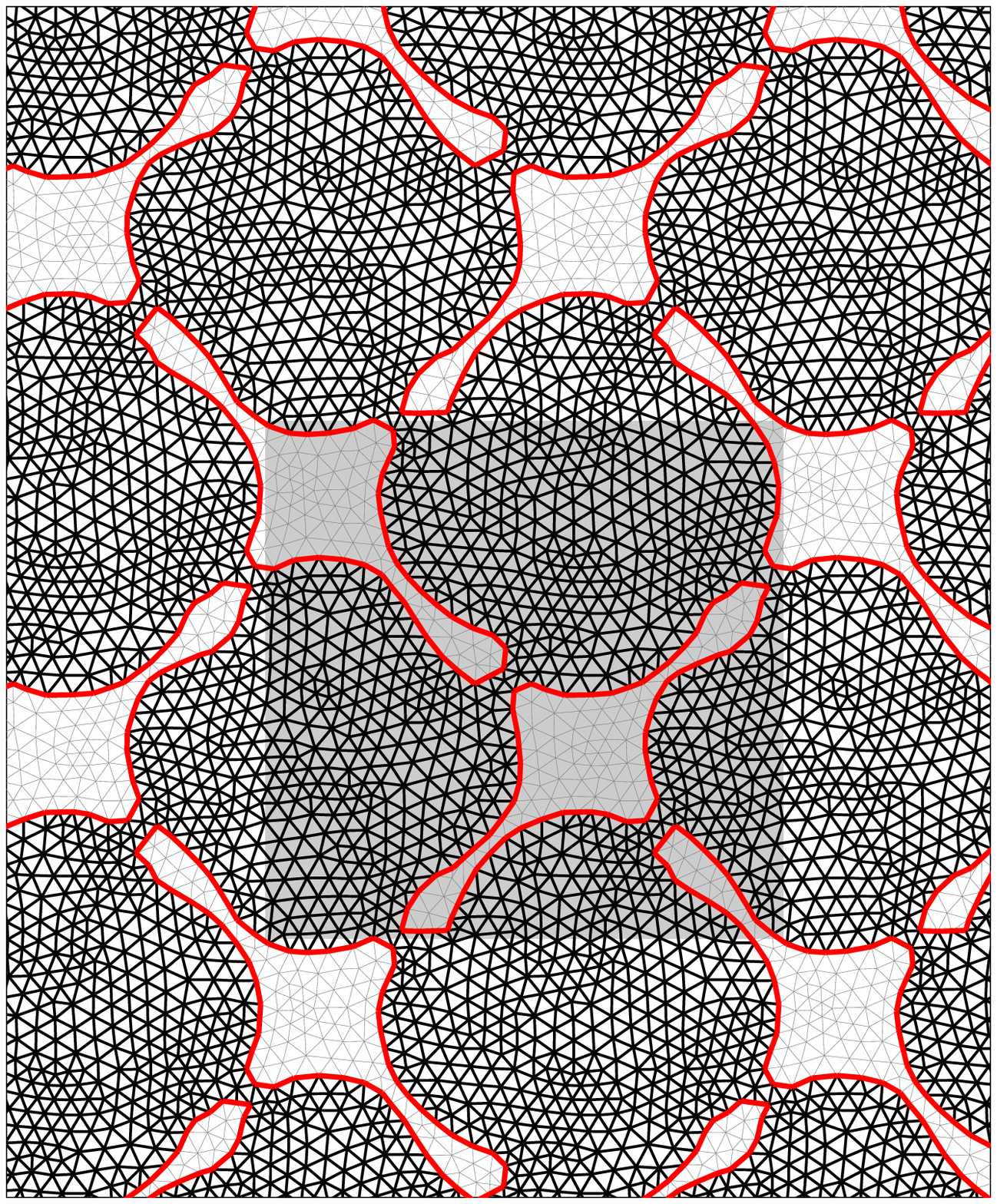}}
\vspace{1.5cm}
\caption{Initial guess and final microstructure for square periodicity}
\end{figure}
\begin{figure}
\center{\includegraphics[width=3in]{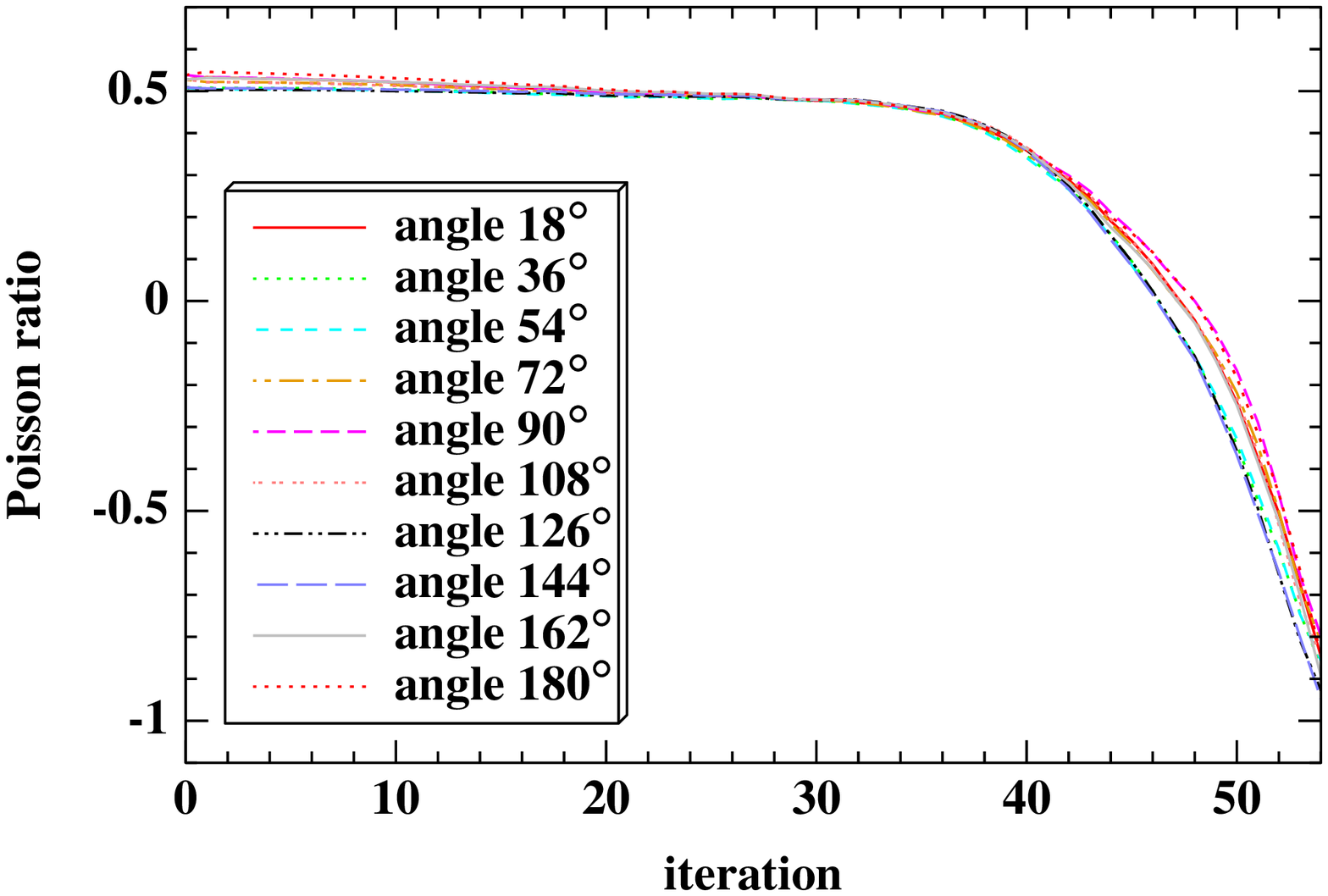}
\includegraphics[width=3in]{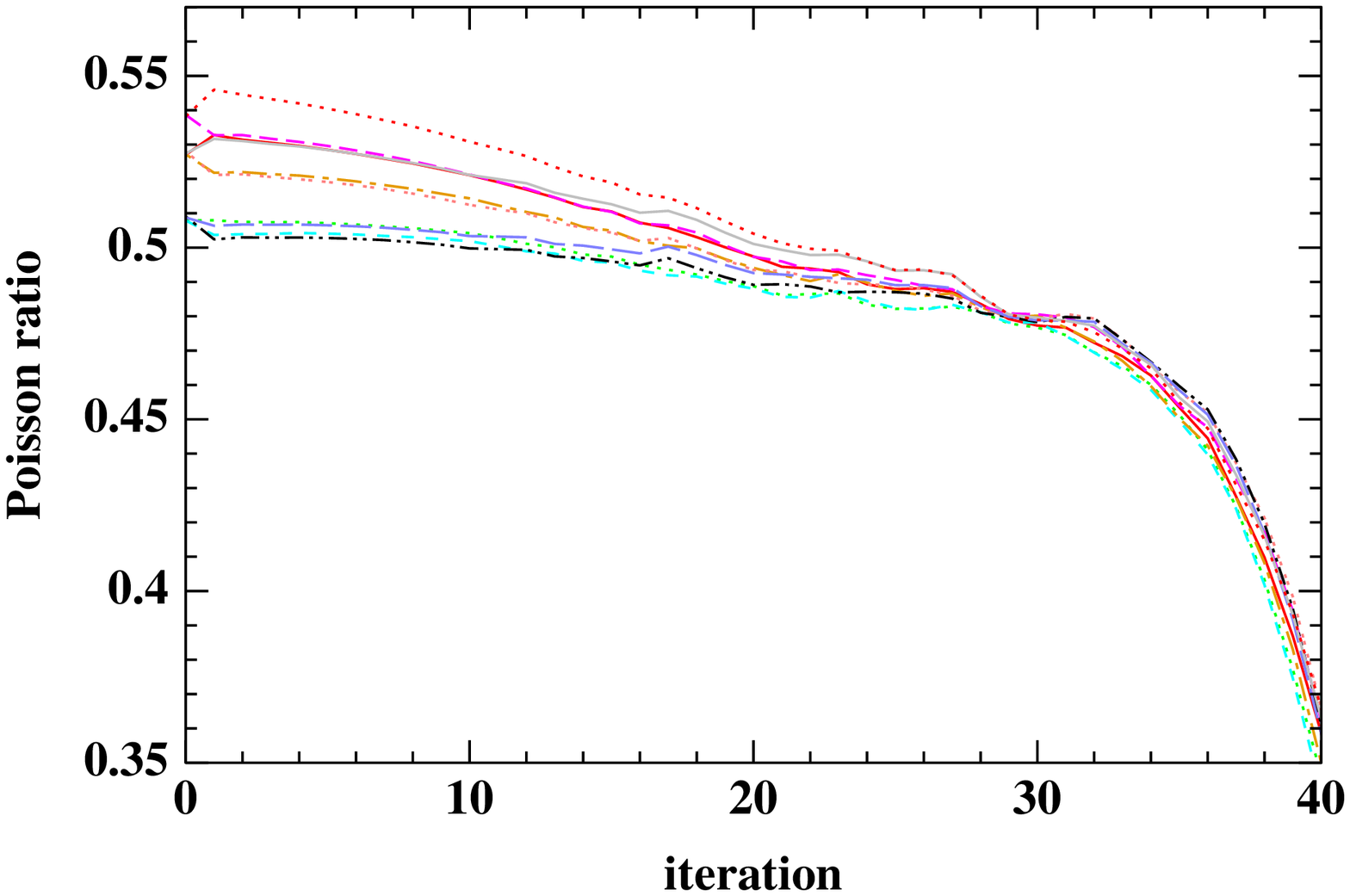}
\includegraphics[width=3in]{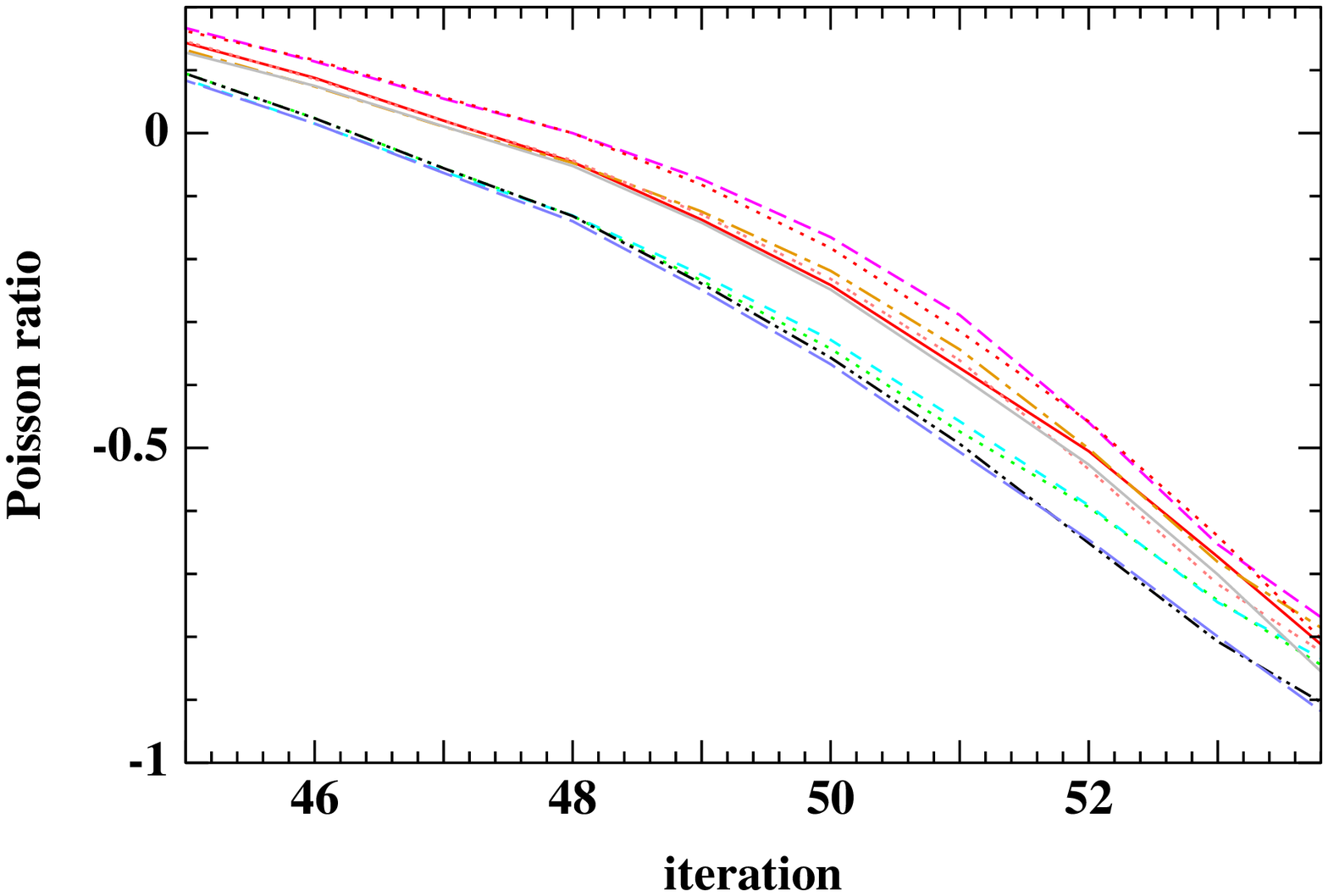}}
\caption{History of convergence, zoom of the first 40 iterations and zoom of the last 6 iterations}
\end{figure}

Topology optimization can be performed by simply changing the triangles
neighbour to a certain vertex from material to void.
However, in the present paper we focus only on shape optimization,
which means that we begin the optimization process with a certain
number of holes in the periodicity cell, and this number is going to remain
constant up to the end of the optimization process.

In the first example we optimize the Poisson ratios along ten directions in the plane, at angles uniformly 
distributed between 0 and 180 degrees. The largest Poisson ratio among those ten directions is minimized.
The initial microstructure (initial guess), see Figure~4, has a square periodicity, property that does not vary during 
the optimization process. It presents two model holes that repeat periodically. 
The algorithm, after 54 iterations, produced the microstructure on the right in Figure~4, still
with square periodicity and still presenting two model holes. At iteration 47 the Poisson ratios became negative
and the final design presents Poisson ratios less then -0.7. In the history of convergence, see Figure~5, 
one can observe that in the first iterations the algorithm makes the ten Poisson ratios as close as possible
and afterwards it decreases them all together. 

In the second example eighteen directions in the plane are chosen, at angles uniformly 
distributed between 0 and 180 degrees. The largest Poisson ratio among those directions is minimized.
The initial guess has a hexagonal periodicity and presents one model hole repeated periodically (with respect to
the hexagonal periodicity), see Figure~6. After 60 iterations the algorithm produces the microstructure presented
in Figure~6, on the right. The history of convergence in shown in Figure~7, with two zoom-in views.
The final design has Poisson ratios close to -0.9 in all eighteen directions.
\begin{figure}[ht]
\center{\hspace{2cm}
\includegraphics[height=5.5cm]{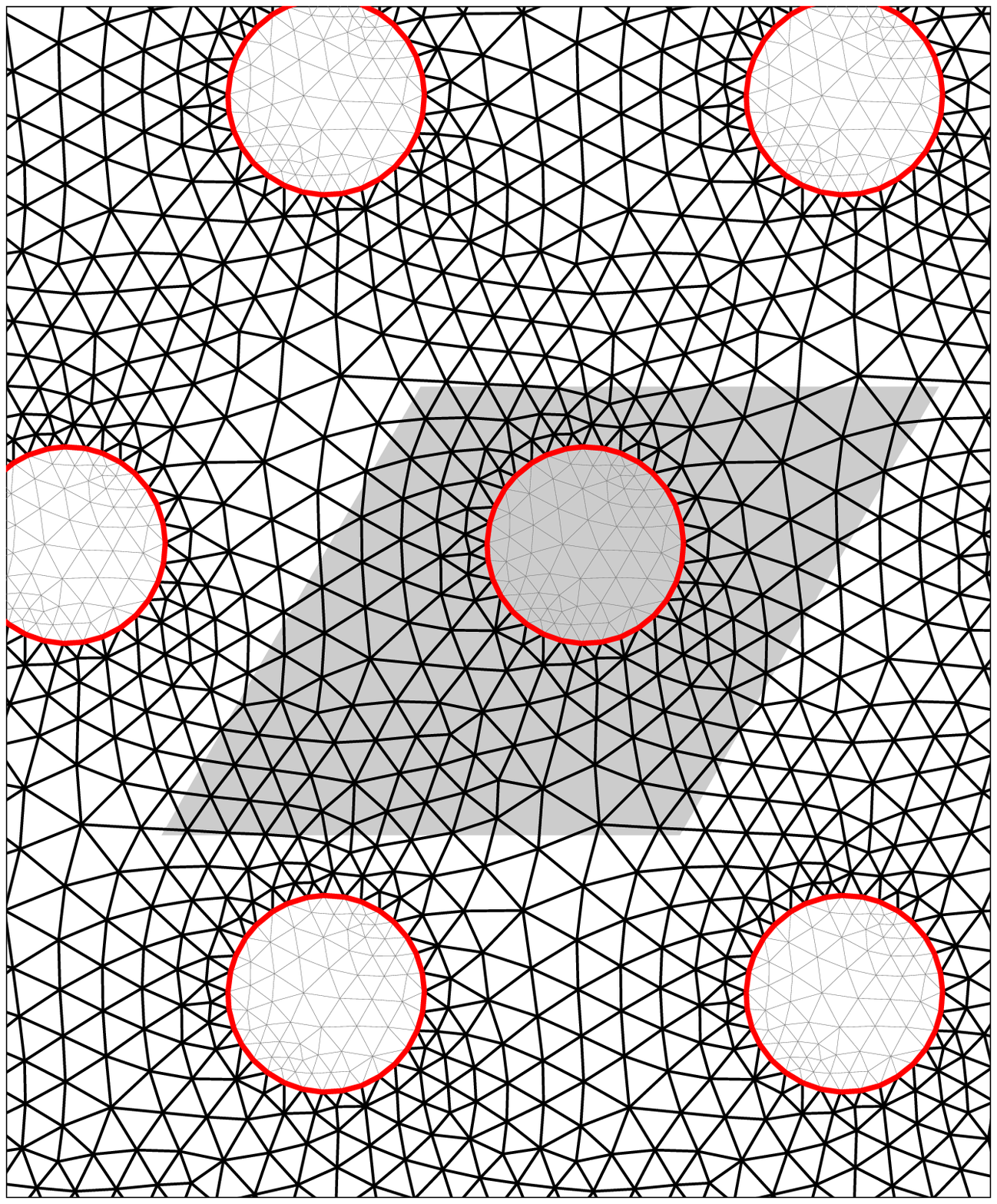}
\includegraphics[height=5.5cm]{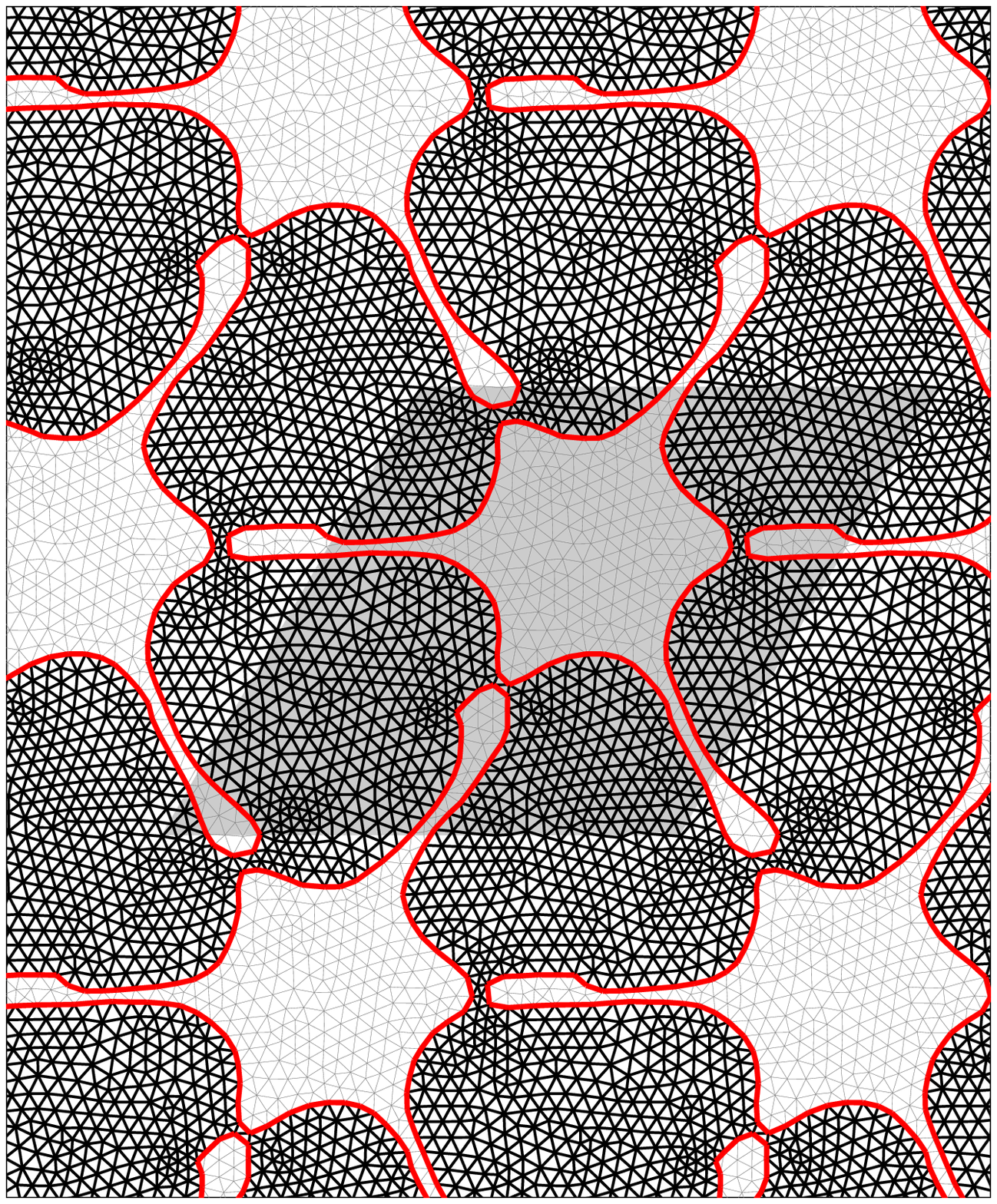}}
\vspace{1.5cm}
\caption{Initial guess and final microstructure for hexagonal periodicity}
\end{figure}
\begin{figure}
\center{\includegraphics[width=4in]{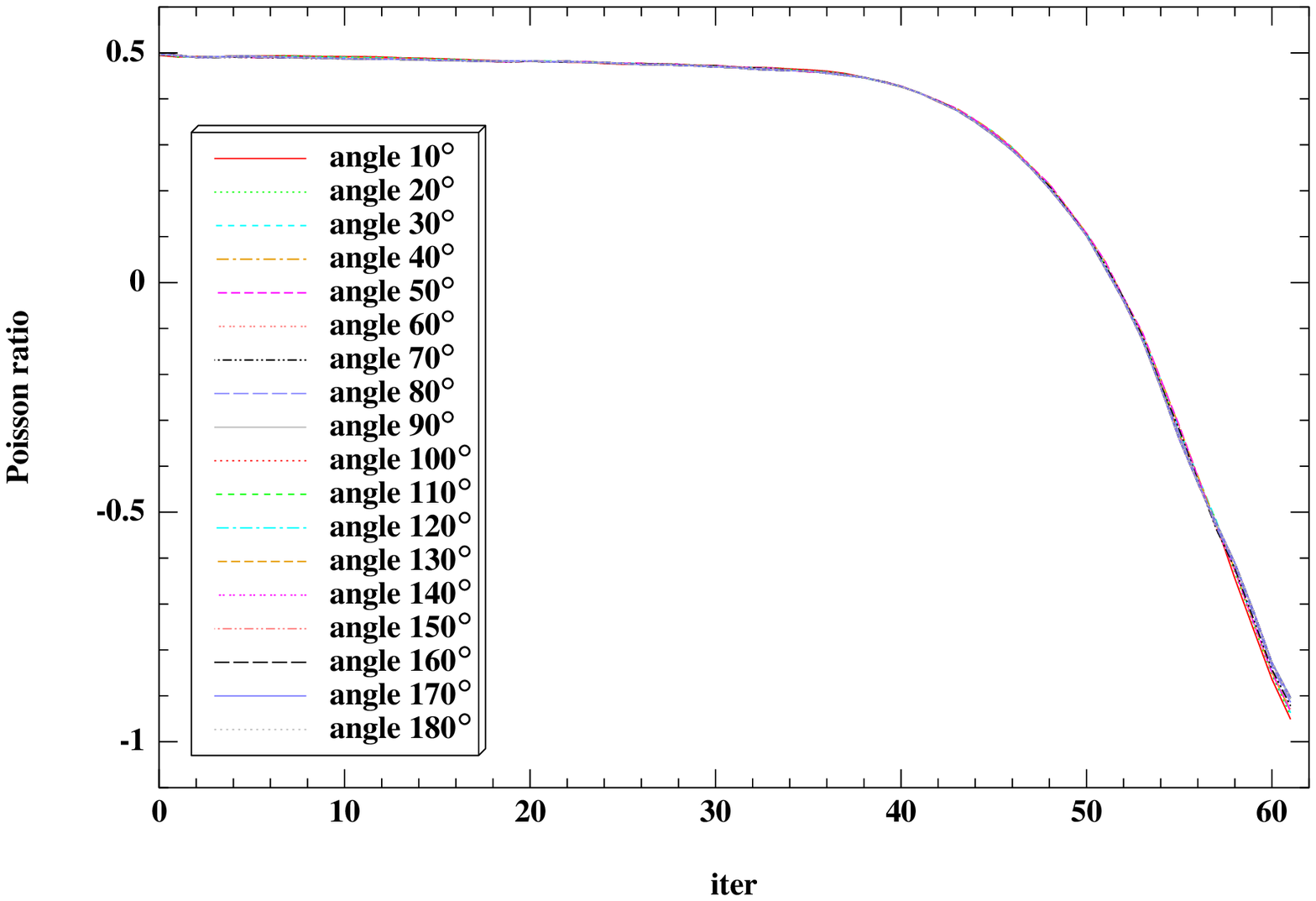}}
\center{\includegraphics[width=3in]{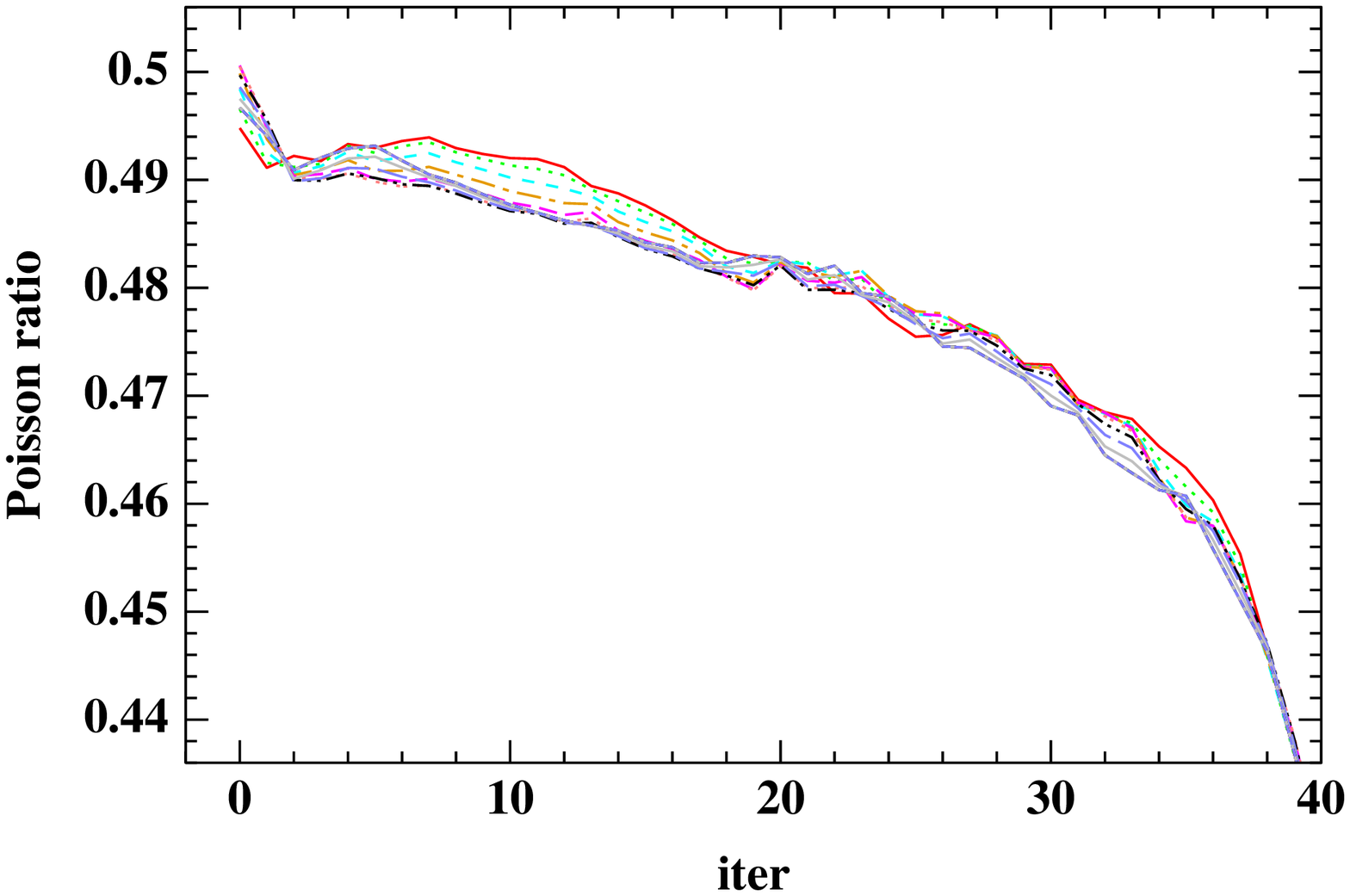}
\includegraphics[width=3in]{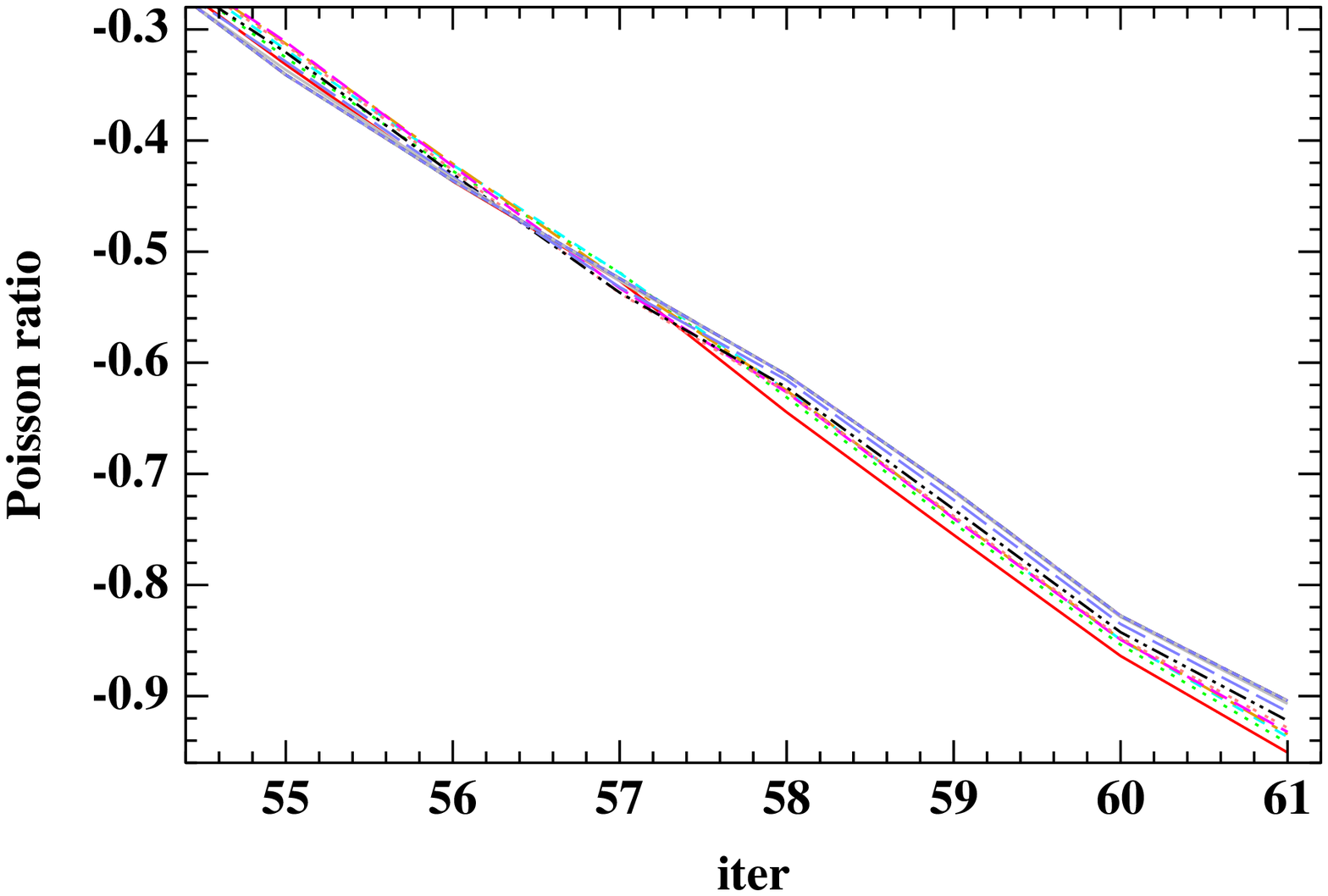}}
\caption{History of convergence, zoom of the first 40 iterations, zoom of the last 8 iterations}
\end{figure}

\section{Conclusions and future development}
\label{conclusion}

An algorithm for optimization with equality constraints
is presented and proven to be convergent.
It can be regarded as a gradient method applied in
the direction tangent to the manifold determined by the constraints, 
together with a Newton method applied in the orthogonal direction. 
This method, although not very fast (it has linear convergence) 
is quite natural, easy to implement, and has the advantage of requiring
solely the first derivatives of the objective and of the constraint
functions. 

Generalizations of this algorithm are described in order to address inequality
constrained problems also, as well as minimax problems.
Criteria for activation and deactivation of constraints are discussed in some
detail, as well as the special case of box-like constraints.
An approach to solve problems involving a (continuous) infinite family
of inequality constraints is also discussed.

An application to optimization of periodic microstructures,
for obtaining homogenized elastic tensors with negative Poisson ratio,
is presented.
It uses shape and/or topology variations in the model hole that
characterizes the microstructure.
Multi-objective optimization is employed in order to minimize 
the Poisson ratio of the homogenized elastic tensor 
in several prescribed directions of the plane, in order to obtain 
a material having roughly the same negative Poisson ratio in all directions.

It is possible to obtain periodic 2D microstructures with Poisson ratio
close to $-1$ and roughly the same in all directions of the plane.
The examples show that the algorithm tends to ``cut'' the structure,
so the bulk modulus approaches zero.

Future work includes the study of three-dimensional microstructures,
the treatment of a continuum of constraints 
and its extension to continuous minimax problems.
When put in the context of optimization of microstructures, the capacity to
deal with a continuum of constraints would allow one to impose as a constraint 
that the composite material be truly isotropic.

\section*{Acknowledgement}

This work was supported by Funda\c c\~ao para a Ci\^encia e a
Tecnologia, UID/MAT/04561/2013.

	\appendix


  \section{Unconstrained optimization}
  \label{sec: unconstr}

    Unconstrained optimization is the minimization or maximization of a 
    scalar function $f$
    defined on the whole $ \RR^n $. 
		


    Minimization algorithms require the user to supply a starting point, which will be denoted by $x^{(0)}$.
    Then, at each iteration, the algorithm chooses a direction
    $\delta^{(k)}$ and searches along this direction, from the current iterate $x^{(k)}$, for a new iterate with a lower
    function value. The distance to move along $\delta^{(k)}$, the \emph{step length}, is commonly chosen after a finite
    number of trial step lengths; this strategy is known as \emph{line search} . This
    kind of procedure is useful for obtaining convergence from ``remote'' initial approximations $x^{(0)}$, which
    is not our main concern. Besides, once a locally convergent algorithm has been devised, with the step length taken
    to be constant throughout, one can always modify it to encompass line search in order to enhance its convergence
    properties (this is usually the order things are done anyway).

    It is quite natural to look for a \emph{descent direction}, that is, a direction $\delta^{(k)}$ such that 
    $ \dual{\nabla f(x^{(k)})}{\delta^{(k)}} < 0 $, where $ \dual{\phantom{i}}{\phantom{i}} $ denotes the usual dot product 
    in $\RR^n$.
    The \emph{steepest descent direction} $ \delta^{(k)} = - \nabla f(x^{(k)}) $ is the most obvious choice.
    This \emph{steepest descent method} has the advantage of requiring the calculation of
    first derivatives only, but it can be quite slow.

    The following standard results will be used 
    (they can be easily found in textbooks on Functional Analysis):
    \begin{theorem}[\textbf{Banach fixed-point theorem}]
      Assume that $K$ is a nonempty closed set in a Banach space $E$ (with norm $\|\phantom{x}\|$), and further, that
      $ S : K \to K $ is a contractive mapping (\textit{i.e.}, a Lipschitzian mapping with Lipschitz constant $L$
      strictly lower than one). Then there exists a unique $ x^\ast \in K $ such that $ x^\ast = S(x^\ast) $ and, for
      any $ x^{(0)} \in K $, the sequence $(x^{(k)})$ defined by $ x^{(k+1)} = S(x^{(k)}) $, $ k \in \NN_0 $, stays in $K$ and
      converges to $x^\ast$. Furthermore, the following estimate holds: $ \|x^{(k)} - x^\ast\| \leqslant
      L^k \|x^{(0)} - x^\ast\| $, for all $ k \in \NN_0 $.
    \end{theorem}

    \begin{corollary}\label{cor: contractivity}
      Let $ S : E \to E $ be a continuously Fr\'echet differentiable operator and $ x^\ast \in E $ a point such that
      $ S(x^\ast) = x^\ast $. If the Fr\'echet derivative of $S$ at $ x^\ast $ has operator norm strictly lower than one,
      then the conclusions of the previous theorem hold with $ K = \{ x \in E : \; \|x - x^\ast\| \leqslant r \} $, for
      some $ r > 0 $. In the finite dimensional case $ E = \RR^n $, this is equivalent to the requirement that
      $ \| DS(x^\ast) \| < 1 $ for some natural norm.\footnote{A matrix norm that is associated with a vector norm is
      called a \emph{natural norm}.}
    \end{corollary}

    The classical local convergence result for the steepest descent method is now presented. 
    The assumptions, as well
    as the proof, are somewhat different than the usual ones 
    encountered in most of the literature, in the sense that
    we regard the method as a fixed-point iteration.
 This choice suits best our reasoning for the convergence proof 
 in subsection~\ref{sec: convergence}.
   \begin{theorem}\label{th: descent}
      Assume that $f$ is a twice continuously differentiable function whose 
      Hessian matrix $ D^2\! f(x^\ast) $ at a local minimizer
      $ x^\ast $ is positive definite. 
      Then there exists $ r > 0 $ such that, given $ x^{(0)} \in \bar{B}_r(x^\ast) = \{ x
      \in \RR^n : \; \|x - x^\ast\|_2 \leqslant r \} $, 
      the steepest descent method $ x^{(k+1)} = x^{(k)} - \eta\, \nabla\!
      f(x^{(k)}) $, $ k \in \NN_0 $, converges linearly to $ x^\ast $ 
      for sufficiently small step lengths $ \eta > 0 $.
    \end{theorem}

    \begin{proof}[\bf \textit{Proof}]
      Taking $ S : \RR^n \to \RR^n $ defined as $ S(x) = x - \eta\, \nabla\! f(x) $, 
      the steepest descent method becomes
      $ x^{(k+1)} = S(x^{(k)}) $, $ k \in \NN_0 $. 
      Since we are not interested in proving global convergence, the
      contractivity property will not be needed in all of $\RR^n$, but only locally near $x^\ast$. 
      By Corollary~\ref{cor: contractivity}, it suffices to check that
      $ \| D S(x^\ast)\| < 1 $ for some natural norm.

      It is clear that $ D S(x^\ast) = I - \eta\, D^2\! f(x^\ast) $ 
      is a symmetric matrix; then we know that the $\ell_2$ norm of 
      $ D S(x^\ast)$ coincides with the spectral radius of this
      same matrix (see \cite[Section 1.4]{PC90}). 
      The eigenvalues of $ D S(x^\ast)$ take the form $ 1 - \eta\, \mu^\ast_i $
      ($1 \leqslant i \leqslant n$), where $ \mu^\ast_1 \geqslant \cdots 
      \geqslant \mu^\ast_n $ are the eigenvalues of $ D^2\! f(x^\ast) $; 
      given that the latter are all positive, 
      we have $ 1 - \eta\, \mu^\ast_i \in [1-\eta\,\mu^\ast_1,
      1-\eta\,\mu^\ast_n] $ ($ 1 \leqslant i \leqslant n $) and 
      the choice $ 0 < \eta < \frac{2}{\mu^\ast_1} $ implies
      that $ [1-\eta\,\mu^\ast_1, 1-\eta\,\mu^\ast_n] \subset \; ]-1,1[ $. 
      Hence, one gets $ \| D S(x^\ast) \|_2 =
      \rho(DS(x^\ast)) $ strictly lower than one.
    \end{proof}

  \section{Optimization under equality constraints}
  \label{sec: constr}

Constrained optimization can be viewed as the superposition of a miniminization
problem and a non-linear equation.
Besides the \emph{objective function} $ f:\RR^n \to \RR $ that we want to minimize, 
a \emph{constraint function} $ g:\RR^n\to \RR^m $ is given defining certain
equations that the unknown vector $x$ must satisfy. 
Thus, the optimization problem can be written 
(considering only equality constraints) as :
    \[ \min_{x \in \M} f(x), \qquad \M = \{ x \in \RR^n : \; g(x) = 0
        \}. \tag{$\mathcal P$} \] 

Constrained optimization problems arise from models involving, for instance,
budgetary limitations or other specifications on the design.
Unconstrained optimization problems arise directly in many practical applications,
and also as reformulations of constrained ones, 
if the constraints are replaced by
penalization terms added to the objective function and having 
the effect of discouraging violations of the constraints, or by other means 
(\textit{e.g.} by parametrizing the set $\M$).
    
    \begin{definition}\label{regular point}
      A point $ x \in \RR^n $ satisfying the constraint $ g(x) = 0 $ is said to be a \emph{regular point} if
      the gradient vectors $ \nabla g_1(x), \nabla g_2(x), \ldots, \nabla g_m(x) $ are linearly
      independent. In other words, the Jacobian matrix $ Dg(x) $ should have full rank (equal to $m$).
    \end{definition}

    Note that at a regular point $x$ the constraint function $g$ 
    is a submersion, giving $\M$ the appropriate geometrical concept, 
    namely that of a submanifold of $\RR^n$; 
    the tangent subspace to $\M$ is given by 
    $ \mathcal{T}_x = \{ \tau \in \RR^n : \; Dg(x)\, \tau = 0 \} $.
    Note also that $ m < n $; in fact, $ m \geqslant n $ 
    would yield a discrete set of feasible points, 
    a situation which is outside the scope of the present paper.

    The \emph{optimality conditions} for constrained optimization problems
    are more complicated than for the unconstrained case.

    \begin{theorem}
    If $ x^\ast \in \RR^n $ is a solution of ($\mathcal P$) and 
    $ x^\ast $ is a regular point, then there exists a unique 
    $ \lambda^\ast \in \RR^m $ (called the \emph{Lagrange multiplier}) 
    such that the following conditions hold:
      \begin{equation}\label{th: kkt}
		\begin{cases} \nabla\! f(x^\ast) + \nabla g(x^\ast)\, \lambda^\ast = 0, \\ g(x^\ast) = 0. \end{cases}
	\end{equation}
    In  coordinate notation:
      \[ \begin{cases} f_{,j}(x^\ast) + \sum_{i=1}^m \lambda_i^\ast\, g_{i,j}(x^\ast) = 0,
        \: 1 \leqslant j \leqslant n, \\ g_i(x^\ast) = 0, \: 1 \leqslant i \leqslant m.
      \end{cases} \]
    \end{theorem}

    These equations are often referred to as \emph{Karush-Kuhn-Tucker conditions}
    or \emph{KKT conditions} for short. 
    They are necessary for optimality, but not sufficient.

    A sufficient optimality condition can be given by the action of the 
    Hessian matrices $ D^2\! f(x^\ast) $ and $ D^2 {g_i}(x^\ast) $ of 
    $f$ and $g_i$ ($ 1 \leqslant i \leqslant m $), respectively, 
    over tangent vectors to $\M$ at $x^\ast$. 
    This sufficient optimality condition involves also the values of
    the Lagrange multipliers $ \lambda^\ast_i $ ($ 1 \leqslant i \leqslant m $).
    The following result can be found in many textbooks on optimization
    (see, for instance, \cite[Section 11.4]{BS03}).

    \begin{theorem}\label{th: existence}
      Suppose there are $ x^\ast \in \RR^n $ and $ \lambda^\ast \in \RR^m $ 
      such that the KKT conditions \eqref{th: kkt} hold.
      Suppose also that the matrix $ H^\ast = D^2\! f(x^\ast) + 
      \sum_{i=1}^m \lambda^\ast_i\, D^2 {g_i}(x^\ast) $ is positive
      definite on $\mathcal{T}_{x^\ast}$, that is, for any nonzero
      vector $ \tau $ tangent to $ \M $,
      there holds $ \dual{H^\ast \tau}{\tau} > 0 $. 
      Then $x^\ast$ is a strict local minimizer of ($\mathcal P$).
    \end{theorem}

\section{Newton's method}
  \label{sec: newton}

Let $ g : \RR\to\RR $ be a (non-linear) function and consider the problem
of finding a root of $g$, that is, of solving the equation $ g(x) = 0 $.
One well-known method is Newton's method (also known under the name of
Newton-Raphson). It consists of starting with some initial guess $ x^{(0)} 
\in\RR $ and defining the sequence $ x^{(k)} $ by iterating
\begin{equation}\label{eq: newton-1}
x^{(k+1)} = x^{(k)} - \frac{g(x^{(k)})}{g'(x^{(k)})}
\end{equation}

\begin{theorem}
Suppose $g$ is twice continuously differentiable and let $ x^\ast\in\RR $ such
that $ g(x^\ast) = 0 $ and $ g'(x^\ast) \neq 0 $. Then, if $ x^{(0)} $ is 
sufficiently close to $ x^\ast $, the sequence $ x^{(k)} $ defined by 
(\ref{eq: newton-1}) converges to $ x^\ast $.
Moreover, the convergence is quadratic.
\end{theorem}

The same basic idea can be applied to systems of (non-linear) equations :

\begin{theorem}
Suppose $ g : \RR^n \to \RR^n $ is twice continuously differentiable and let 
$ x^\ast\in\RR^n $ such that $ g(x^\ast) = 0 $ and $ Dg(x^\ast) $ is invertible. 
Then, if $ x^{(0)} $ is sufficiently close to $ x^\ast $, the sequence $ x^{(k)} $ 
defined by 
\begin{equation}\label{eq: newton-syst}
x^{(k+1)} = x^{(k)} - \bigl(Dg(x^{(k)})\bigr)^{-1}\,g(x^{(k)})
\end{equation}
converges to $ x^\ast $.
Moreover, the convergence is quadratic.
\end{theorem}

Newton's method can be extended to solve under-determined equations.
For instance, let $ g : \RR^n\to\RR^m $, with $ m\leqslant n $, be a non-linear function 
and consider the problem of finding a root of $g$, that is, of solving the equation 
$ g(x) = 0 $.
If $ m<n $, this equation is under-determined; the set of solutions
will be a manifold $ \M $ in $ \RR^n $ of dimension $ n-m $.
However, it may be of interest to solve numerically this equation, that is,
to start with some $ x^{(0)} \in\RR^n $ and to build iteratively a sequence
$ x^{(k)} $ which converges to some $ x^\ast $ such that $ g(x^\ast) = 0 $.
This can be done, roughly speaking, by building a step orthogonal to the
level set of $g$ and obeying, within that orthogonal subspace, to the basic idea
of the Newton method.
The following result describes the procedure in detail.

\begin{theorem}\label{th: newton}
Let $ g : \RR^n \to \RR^m $ ($ m \leqslant n $) be differentiable, $Dg$ be of full
rank $m$ in an open convex set $D$ and let the following hold:
\begin{itemize}
\item[\bf (i)] there exists $ K \geqslant 0 $ and $ \alpha \in {]}0,1] $ such that
$ \| Dg(y) - Dg(x) \| \leqslant K \|y - x\|^\alpha $ for all $ x,y \in D $; 
\item[\bf (ii)] there is a constant $B$ for which $ \| Dg(x)^+ \| \leqslant B $ for all
$ x \in D $, where $ Dg(x)^+ = \nabla g(x)\,[Dg(x)\, \nabla g(x)]^{-1} $ is the
Moore-Penrose inverse of $Dg(x)$.
\end{itemize}
Furthermore, for $ \eta > 0 $, let
$ D_\eta = \left\{ x \in D : \, \|y - x\| < \eta \Rightarrow y \in D \right\} $.
Then there exists an $ \epsilon > 0 $ depending only on $K$, $\alpha$, $B$ and $\eta$ such that
if $ x^{(0)} \in D_\eta $ and $ \| Dg(x^{(0)}) \| < \epsilon $, then the iterates $ x^{(k)} $
determined by
\begin{equation}\label{eq: newton-m}
	x^{(k+1)} = x^{(k)} - Dg(x^{(k)})^+ g(x^{(k)})
\end{equation}
are well defined and converge to a point $ x^\ast \in D $ such that $ g(x^\ast) = 0 $.
Moreover, there is a constant $\beta$ for which
\[ \left\| x^{(k+1)} - x^\ast \right\| \leqslant \beta \left\| x^{(k)} - x^\ast \right\|^{1+\alpha},
	 \quad k = 0,1,2,\ldots \]
\end{theorem}

Note that, if $ m=n $, the iteration (\ref{eq: newton-m}) above reduces to 
equation (\ref{eq: newton-syst}) which defines the sequence $ x^{(k)} $ in the 
determined case;
the rate of convergence is still quadratic if $Dg$ is Lipschitz continuous.
Convergence theory on such Newton-like methods, for under-determined systems of equations,
can be found in \cite{WW90} and references therein.

Since the considered system of equations has many solutions (a manifold of them),
it is legitimate to question whether the sequence $ x^{(k)} $ must be so rigidly defined.
That is, one may ask whether the step $ \delta^{(k)} $ must be necessarily orthogonal
to the level set of $g$.
The following result (whose proof is immediate) shows that the philosophy behind
Newton's method can be applied while allowing for uncertainties in the
definition of the sequence $ x^{(k)} $.

\begin{proposition}\label{prop: Newton}
Let $ g : \RR^n\to\RR^m $ be continuously differentiable and consider a sequence 
$ x^{(k)} $ which converges to some $ x^\ast\in\RR^n $ and which satisfies
\begin{equation}\label{eq: newton-n}
Dg(x^{(k)})\,\delta^{(k)} = - g(x^{(k)})
\end{equation}
where $ \delta^{(k)} = x^{(k+1)} - x^{(k)} $.
In coordinate notation,
\begin{equation*}
\sum_{j=1}^n g_{i,j}(x^{(k)})\, \delta_j^{(k)} = - g_i(x^{(k)})\,,\
1\leqslant i\leqslant m
\end{equation*}
Then $ x^\ast $ is a solution, that is, $ g(x^\ast) = 0 $.
\end{proposition}

Note that condition (\ref{eq: newton-m}) implies (\ref{eq: newton-n}).
Note also that, if $ m=n $, then (\ref{eq: newton-syst}), (\ref{eq: newton-m}) and
(\ref{eq: newton-n}) are all equivalent.
However, for $ m<n $, unlike (\ref{eq: newton-syst}) or (\ref{eq: newton-m}), 
equation (\ref{eq: newton-n}) does not define uniquely the sequence $ x^{(k)} $; 
it simply states a property of the sequence.
The user has the freedom to chose $ n-m $ components of
$ \delta^{(k)} $ : those orthogonal to $ \nabla g_1(x^{(k)}) $, $\nabla g_2(x^{(k)}) $,
$ \dots $, $\nabla g_m(x^{(k)}) $, that is, components tangent to the level set of $g$.
The user may use this freedom in order to
solve other equation(s) or to minimize some functional.

\begin{remark}\label{rem: faster}
The quantity $ g(x^{(k)}) $ converges to zero faster than the convergence of
$ x^{(k)} \to x^\ast $.
For instance, suppose $ x^{(k)} $ converges linearly, that is, there is a constant
$ L\in {]}0,1{[} $ such that $ \|x^{(k)}-x^\ast\| $ is of order $ O(L^k) $. 
Then $ \delta^{(k)} = x^{(k+1)}-x^{(k)} $ is also of order $ O(L^k) $ and
a simple Taylor expansion about $x^{(k)}$ yields
\[ g(x^{(k+1)}) = \underbrace{g(x^{(k)}) + Dg(x^{(k)})\,\delta^{(k)}}
_{= \: 0, \text{ due to equation (\ref{eq: newton-n})}}
 + \; O(\|\delta^{(k)}\|^2); \] 
thus, $ g(x^{(k)}) $ is of order $ O(L^{2k}) $.
Note however that this is not quadratic convergence, but simply an improved linear one.
Nor should quadratic convergence be expected because, unlike in
Theorem~\ref{th: newton}, here the step $ \delta^{(k)} = x^{(k+1)} - x^{(k)} $
is not a full Newton-type step.
\end{remark}

\begin{remark}\label{rem: distance}
It can be proven that, locally around a regular point $ x^\ast \in \M = 
\left\{ x \in \RR^n :\, g(x) = 0 \right\} $ (see Definition~\ref{regular point} in \ref{sec: constr}), 
the quantity $ \|g(x^{(k)})\| $ is of the same order of magnitude as the distance 
$ \mathrm{dist}(x^{(k)},\M) $ to the manifold $ \M $.
This geometric property, taken together with Remark \ref{rem: faster} above, 
implies that the distance $ \mathrm{dist}(x^{(k)},\M) $
converges to zero faster than the convergence of $ x^{(k)} \to x^\ast $.
\end{remark}



\bibliographystyle{elsarticle-num}
\bibliography{<your-bib-database>}

\begin{thebibliography}{99}



\bibitem{Bar03}
C.~Barbarosie, \emph{Shape optimization of periodic structures},
Computational Mechanics, \textbf{30}, 235--246, 2003.

\bibitem{JB77}
J.T.~Betts, \emph{An Accelerated Multiplier Method for Nonlinear Programming},
Journal of Optimization Theory and Applications, \textbf{21}(2), 137--174, 1977.

\bibitem{JR61}
J.~Rosen, \emph{The Gradient Projection Method for Nonlinear Programming, II. Nonlinear Constraints},
Journal of the Society for Industrial and Applied Mathematics, \textbf{9}, 514--532, 1961.

\bibitem{BL11}
C.~Barbarosie, S.~Lopes,
\emph{A gradient-type algorithm for optimization with constraints},
preprint CMAF Pre-2011-001, available at {\tt http://cmaf.fc.ul.pt/preprints.html}

\bibitem{PC90}
P.G.~Ciarlet, \emph{Introduction \`a l'Analyse Num\'erique Matricielle et \`a l'Optimisation}, Masson, 1990.


\bibitem{BS03}
J.~Bonnans, J.~Gilbert, C.~Lemar\'echal, C.~Sagastiz\'abal, \emph{Numerical Optimization -- Theoretical and
Practical Aspects}, Springer, 2003.

\bibitem{NW06}
J.~Nocedal, S.~Wright, \emph{Numerical Optimization (Second Edition)},
Springer, 2006.

\bibitem{BLcompl}
C.~Barbarosie, S.~Lopes, \emph{A generalized notion of compliance},
Comptes Rendus M\'ecanique, \textbf{339}, 641--648, 2011.

\bibitem{BarToa13}
C.~Barbarosie, A.-M.~Toader, 
Multi-objective optimization of composite materials with negative Poisson ratio along several directions,
preprint CMAF Pre-011, 2013, available at {\tt http://cmaf.fc.ul.pt/preprints.html}

\bibitem{Bar08a}
C.~Barbarosie, A.-M.~Toader, 
\emph{Shape and Topology Optimization for periodic problems,
Part I, The shape and the topological derivative},
Structural and Multidisciplinary Optimization, \textbf{40}, 381--391, 2010.

\bibitem{Toa11}
A.-M.~Toader, \emph{The topological derivative of homogenized elastic 
coefficients of periodic microstructures},
SIAM Journal on Control and Optimization, \textbf{49}(4), 1607--1628, 2011.

\bibitem{Bar08b}
C.~Barbarosie, A.-M.~Toader, 
\emph{Shape and Topology Optimization for periodic problems,
Part II, Optimization algorithm and numerical examples},
Structural and Multidisciplinary Optimization, \textbf{40}, 393--408, 2010.

\bibitem{BSS93}
A.~Boresi, R.~Schmidt, O.~Sidebottom,
\emph{Advanced Mechanics of Materials}, Wiley, 1993.

\bibitem{WW90}
H.~Walker, L.~Watson,
\emph{Least-change secant update methods for underdetermined systems},
SIAM Journal of Numerical Mathematics, \textbf{27}(5), 1227--1262, 1990.

\end{thebibliography}



\end{document}